\documentclass[reqno,12pt]{amsart}
\usepackage{amsmath,amsfonts, amstext, amsthm, amssymb,color, epsfig,indentfirst}
\usepackage{float}

\numberwithin{equation}{section}
\newtheorem{thm}{Theorem}[section]

\newtheorem{cor}{Corollary}[section]

\newtheorem{rem}{Remark}[section]

\newcommand{\Z}{\mathbb{Z}}

\newcommand{\ds}{\operatorname{ds}}

\def\ds{\displaystyle}

\begin{document}
\title[Product Identities for
Theta Functions]{Integer Matrix Exact Covering Systems and Product
Identities for Theta Functions}

\author{Zhu ~Cao}
\address{Department of Mathematics, University of Mississippi, University, MS 38677,  USA}
\email{zcao3@olemiss.edu}

\begin{abstract}
In this paper, we prove that there is a natural correspondence
between product identities for theta functions and integer matrix
exact covering systems. We show that since $\Z^n$ can be taken as
the disjoint union of a lattice generated by $n$ linearly
independent vectors in $\Z^n$ and a finite number of its translates,
certain products of theta functions can be written as linear
combinations of other products of theta functions. We firstly give a
general theorem to write a product of $n$ theta functions as a
linear combination of other products of theta functions. Many known
identities for products of theta functions are shown to be special
cases of our main theorem. Several entries in Ramanujan's notebooks
as well as new identities are proved as applications, including
theorems for products of three and four theta functions that have
not been obtained by other methods.

\end{abstract}

\maketitle

\vskip2pt

\bigskip
\section{Introduction and Main Theorem}
We use the standard notation for $q$-products, defining
\begin{align} \label{qprod}
(a)_\infty := (a;q)_{\infty} = \prod_{k=0}^{\infty} (1 - aq^{k}),
\quad |q| < 1.
\end{align}
The celebrated Jacobi triple product identity is given by \cite[p.
$10$]{spirit} {\allowdisplaybreaks
\begin{align} \label{JTI}
\sum_{n= -\infty}^{\infty} q^{n^{2}} z^{n} = (-qz;q^{2})_{\infty}
(-q/z;q^{2})_{\infty} (q^{2}; q^{2})_{\infty}, \qquad |q| < 1.
\end{align}
We define the modified Jacobi theta function $\langle
x;q\rangle_\infty:=(q;q)_\infty(x;q)_\infty(q/x;q)_\infty$. So the
Jacobi triple product identity can be written as
\begin{align*}
\sum_{n= -\infty}^{\infty} q^{n^{2}} z^{n} =\langle
-qz;q^2\rangle_\infty , \qquad |q| < 1.
\end{align*}
We use abbreviated forms for $q$-products
\begin{align*}
(\alpha,\beta,\dots,\gamma; q)_\infty
&:=(\alpha;q)_\infty(\beta;q)_\infty \dots(\gamma;q)_\infty,\\
\langle\alpha,\beta,\dots,\gamma; q\rangle_\infty
&:=\langle\alpha;q\rangle_\infty\langle\beta;q\rangle_\infty
\dots\langle\gamma;q\rangle_\infty.
\end{align*}
Ramanujan's general theta function is defined as
\begin{align}
f(a,b):=\sum^\infty_{n=-\infty}a^{n(n+1)/2}b^{n(n-1)/2}, \ |ab|<1.
\end{align}
By Jacobi's triple product identity, we have
\[f(a,b)=(-a; ab)_\infty(-b; ab)_\infty(ab; ab)_\infty.\]
It is easy to verify that
\begin{align}
f(a, b)&=f(b, a),\\
f(1, a)&=2f(a, a^3),\\
\label{f-1a}
f(-1, a)&=0,
\end{align}
and, if $n$ is an integer,
\begin{align} \label{aa}
f(a, b)=a^{n(n+1)/2}b^{n(n-1)/2}f(a(ab)^n, b(ab)^{-n}).
\end{align}
The three most important cases of $f(a, b)$, namely
 {\allowdisplaybreaks
\begin{align*} \varphi(q) :&=\sum_{n=-\infty}^{\infty} q^{n^2} = f(q,
q) =(-q;q^{2})^2_{\infty}(q^{2}; q^{2})_{\infty},\\
\psi(q) :&=\sum_{n=0}^{\infty} q^{n(n +1)/2} = f(q, q^3) =
\frac{(q^{2}; q^{2})_{\infty}}{(q;q^{2})_{\infty}},
\\
f(-q):&=f(-q, -q^2)=(q; q)_\infty,
\end{align*}
are also used throughout this paper. We also define $\chi(q)=(-q;
q^2)_\infty$, which is not a theta function. We define it mainly for
convenience.

Define $(a)_n: =a(a+1)(a+2)\cdots(a+n-1)$. A modular equation of
degree $n$ is an algebraic relation of $\alpha$ and $\beta$ which is
induced by the relation
\begin{align} \label{defmodular}
\frac{_2F_1(\frac{1}{2}, \frac{1}{2}; 1;
1-\beta)}{_2F_1(\frac{1}{2}, \frac{1}{2}; 1;
\beta)}=n\frac{_2F_1(\frac{1}{2}, \frac{1}{2}; 1;
1-\alpha)}{_2F_1(\frac{1}{2}, \frac{1}{2}; 1; \alpha)},
\end{align}
 where
\[_2F_1(a, b; c; z)
:=\sum_{n=0}^\infty\frac{(a)_n(b)_n}{(c)_n n!}{z^n}, \qquad |z| <
1,\] is the Gaussian or ordinary hypergeometric functions. When
\eqref{defmodular} holds, we say that $\beta$ has degree $n$ over
$\alpha$.

A system of congruences $a_i\pmod {n_i}$ with $1\leq i \leq k$ is
called a covering system (or complete residue system) if every
integer $y$ satisfies $y \equiv a_i \pmod {n_i}$ for at least one
value of $i$.

A covering system in which each integer is covered by just one
congruence is called an exact covering system (ECS). In other words,
an exact covering system is a partition of the integers into a
finite set of arithmetic sequences.

Corresponding to the exact covering system $\big\{r\,
(\text{mod}\,\, {k})\big\}_{r=0}^{k-1}$, we can write any theta
function as the linear combination of $k$ theta functions
\begin{align}
f(a,
b)=&\sum^\infty_{n=-\infty}a^{n(n+1)/2}b^{n(n-1)/2}=\sum^{k-1}_{r=0}\sum^\infty_{n=-\infty}a^{(kn+r)(kn+r+1)/2}b^{(kn+r)(kn+r-1)/2}
\notag \\  \label{entry31}
=&\sum^{k-1}_{r=0}a^{r(r+1)/2}b^{r(r-1)/2}f(a^{k(k+1)/2+kr}b^{k(k-1)/2+kr},a^{k(k-1)/2-kr}b^{k(k+1)/2-kr}).
\end{align}
If we define $U_k=a^{k(k+1)/2}b^{k(k-1)/2}$,
$V_k=a^{k(k-1)/2}b^{k(k+1)/2}$, then we can write \eqref{entry31} as
\begin{align} \label{entry 31}
f(a,b)=f(U_1, V_1)=\sum^{k-1}_{r=0}U_rf(\frac{U_{k+r}}{U_r},
\frac{V_{k-r}}{U_r}),
\end{align}
This is \cite[pp.~48--49, Entry 31]{Ber1}. Letting $k=2$ in
\eqref{entry 31}, we have
\begin{align} \label{entry 311}
f(a,b)=f(a^3b, ab^3)+af(b/a, a^5b^3),
\end{align}
which is used in this paper.

A natural question is: For a product of $n$ $(n\geq 2)$ theta
functions, do we have similar results as \eqref{entry 31}? In other
words, can we write a product of $n$ theta functions as a linear
combination of other products of theta functions? The quintuple
product identity and the septuple product identity are examples for
products of two theta functions.
\begin{thm} [The Quintuple Product Identity]
For $a \neq 0$,
\begin{align} \label{qpi}
&(-aq;q)_\infty(-1/a;q)_\infty(a^2q;q^2)_\infty(q/a^2;q^2)_\infty(q;q)_\infty\notag
\\
=&(a^3q^2;q^3)_\infty(q/a^3;q^3)_\infty(q^3;q^3)_\infty+a^{-1}(a^3q;q^3)_\infty(q^2/a^3;q^3)_\infty(q^3;q^3)_\infty.
\end{align}
\end{thm}
\begin{thm} [The Septuple Product Identity]
For $a \neq 0$,
\begin{align} \label{spi}
&\langle a; q^2\rangle_\infty\langle a^2;q^2\rangle_\infty\notag
\\
=&\langle q^4; q^{10}\rangle_\infty \Big[\langle
a^5q^2;q^{10}\rangle_\infty+a^3\langle
a^5q^8;q^{10}\rangle_\infty\Big]-\langle q^2; q^{10}\rangle_\infty
\Big[a\langle a^5q^4; q^{10} \rangle_\infty+a^2\langle
a^5q^6;q^{10}\rangle_\infty\Big].
\end{align}
\end{thm}
For the history of the quintuple product identity, readers can refer
to S. Cooper's survey \cite{cooper}. The septuple product identity
was first discovered by M.~D.~Hirschhorn \cite[(3.1)]{hir3} in 1983.

Winquist's identity is an example for products of four theta
functions.
\begin{thm} [Winquist's Identity]
For any nonzero complex numbers $a, b$,
\begin{align}
&\sum^\infty_{m=-\infty}\sum^\infty_{n=-\infty}(-1)^{m+n}{q^{\frac{3m^2+3n^2+3m+n}{2}}}
\big(a^{-3m}b^{-3n}-a^{-3m}b^{3n+1}-a^{-3n+1}b^{-3m-1}+a^{3n+2}b^{-3m-1}\big)
\notag \\ \label{winquist} =&(a, q/a,b, q/b,ab, q/ab,a/b,bq/a,q, q
 ; q)_\infty.
\end{align}
\end{thm}
Winquist's identity plays a vital role in the first elementary proof
of Ramanujan's congruence $p(11n+6)\equiv 0 \pmod{11}$ given by L.~
Winquist in \cite{win1}, where $p(n)$ denote the partition function.
Readers can refer to its latest proof given by S.~H. Chan
\cite{winquist} for its brief history.

In \cite{hir1}, Hirschhorn obtained a generalization of Winquist's
identity by multiplying four triple products. He considered a very
nice transformation matrix and derived a four-parameter identity
which gives Winquist's identity as a special case. Inspired by
Hirschhorn's work, S.-S.~Huang \cite{huang1} showed that
Hirschhorn's generalization of the quintuple product identity
\cite{hir2} can also be obtained by using the idea in \cite{hir1}.
Other cases of this kind of transformation are discussed in
\cite[pp.~190--191]{Ber2}, \cite{hahn1}, and \cite{Baruah}. But all
of the above are just special cases; there have been few
systematical study on the conditions under which the product of two
or more theta functions can be written as the linear combination of
other products of theta functions.

In this paper, we consider a class of ``generalized orthogonal''
transformation matrices and give a systematic approach for obtaining
product identities.  Although we give the first general theorem on
the product of any $n$ theta functions, most of the examples
discussed in this paper are for products of two theta functions. Our
main theorem for products of two theta functions is a generalization
of the Schr\"{o}ter formula.  Many known identities for products of
two theta functions, including M. D. Hirschhorn's generalization of
the quintuple product identity, the septuple product identity, some
modular relations for the G\"{o}llnitz-Gordon functions found by
S.-S.~Huang in \cite{huang3}, identities involving septic
Rogers-Ramanujan functions obtained by H. Hahn \cite{hahn1}, a
general theorem by W. Chu and Q. Yan \cite{chu}, and the
Blecksmith-Brillhart-Gerst theorem \cite{BBG1}, are shown to be
special cases of our main theorem. We also derive several new
theorems for products of three or more theta functions as
applications, including an analogue of Winquist's identity and a new
representation of $(q; q)^8_\infty$.

Let $l_i$ $\in \Z^{+}$, $a_ib_i=q^{l_i}$, ($i=1,2, \dots, n$).
Without lose of generality, suppose $l_1\leq l_{2}\cdots \leq l_n$.
We consider a product of $n$ theta functions
\begin{eqnarray} \label{main sum}
S:=\prod_{i=1}^nf(a_i, b_i) =\sum^\infty_{x_1, x_2\cdots,
x_n=-\infty}
{a_1}^{\frac{x_1^2+x_1}{2}}{b_1}^{\frac{x_1^2-x_1}{2}}\cdots{a_n}^{\frac{x_n^2+x_n}{2}}{b_n}^{\frac{x_n^2-x_n}{2}}.
\end{eqnarray}
Next, we change the variables from $x_i$ to $y_i$ $(i=1,2,\dots,n)$
by the transformation $y=Ax$, where $A$ is an integer matrix with
$\det A\neq 0$. Set
\[
x=\left(
\begin{array}{rrrr}
x_1\\
x_2\\
\vdots\\
x_n
\end{array}
\right), \qquad y=\left(
\begin{array}{rrr}
y_1\\
y_2\\
\vdots\\
y_n
\end{array}
\right).
\]
By the inverse formula,
\begin{align} \label{inv}
x=A^{-1}y=\frac{1}{\det A}A^*y,
\end{align}
where $A^*$ is the adjugate of $A$.

For $k\in \Z$, $1\leq k\leq n$, the $k$th determinantal divisor of
$A$, denoted by $d_k(A)$, is defined as the greatest common divisor
of all the $k$ by $k$ determinantal minors of $A$. It is easy to see
that $d_n=\det A$. We let $d_0=1$ for convenience. Note that
$d_k|d_{k+1}$. The invariant factors of $A$ are defined as
\[S_k(A)=\frac{d_k}{d_{k-1}}.\]
By the Smith normal form theorem, we have $\det A=s_1\cdots s_n$. So
\begin{align} \label{detsn}
\frac{\det A}{d_{n-1}(A)}={s_n(A)}.
\end{align}
From \eqref{detsn}, we can rewrite \eqref{inv} as
\begin{align} \label{inv1}
x=\frac{1}{s_n(A)}.\frac{A^*}{d_{n-1}(A)}y,
\end{align}
where $\frac{A^*}{d_{n-1}(A)}$ is an integer matrix. Letting
$\hbox{sgn}(s_n(A))\frac{A^*}{d_{n-1}(A)}=B$, $|s_n(A)|=d$, we can
rewrite \eqref{inv1} as
\begin{align} \label{x y}
x=\frac{1}{d}By.
\end{align}
So we have
\begin{align} \label{congruence equation mod d}
By \equiv 0 \pmod {d}.
\end{align}

We want to write \eqref{main sum} as a linear combination of
products of other theta functions. After replacing the variables
$x_1,x_2,\dots, x_n$ with $y_1,y_2,\dots, y_n$, we have the
restriction \eqref{congruence equation mod d} regarding $y$. So
$y_i$ cannot take all the integer values unless $d=1$. The system of
homogeneous congruences \eqref{congruence equation mod d} is always
consistent since $y\equiv 0 \pmod d$ is a solution. Suppose we
have $k$ solutions to \eqref{congruence equation mod d}. \\
Reduced case: $y\equiv 0\pmod d$.  We can replace $y$ by $dy$ in
\eqref{x y}. So we have $x=By$. Replacing $x$ with $By$ in the
right-hand side of \eqref{main sum}, we obtain the contribution of
this case to $S$
\begin{align}
\sum^\infty_{y_1, \cdots,
y_n=-\infty}&{a_1}^{\frac{(b_{11}y_1+b_{12}y_2+\dots+b_{1n}y_n)^2+(b_{11}y_1+b_{12}y_2+\dots+b_{1n}y_n)}{2}}{b_1}^{\frac{(b_{11}y_1+b_{12}y_2+\dots+b_{1n}y_n)^2-(b_{11}y_1+b_{12}y_2+\dots+b_{1n}y_n)}{2}}\notag
\\
\label{case I}
\cdots&{a_n}^{\frac{(b_{n1}y_1+b_{n2}y_2+\dots+b_{nn}y_n)^2+(b_{n1}y_1+b_{n2}y_2+\dots+b_{nn}y_n)}{2}}{b_n}^{\frac{(b_{n1}y_1+b_{n2}y_2+\dots+b_{nn}y_n)^2-(b_{n1}y_1+b_{n2}y_2+\dots+b_{nn}y_n)}{2}}.
\end{align}

We need the coefficients of all $y_iy_j=0$ $(i\neq j,\,\, i, j =1,
2, \dots, n)$ in order to separate $y_1, y_2, \dots, y_n$ and write
$S$ as a linear combination of products of theta functions. So we
have the requirements
\begin{align} \label{orthogonal}
\left\{
\begin{array}{rcrcrcr}
l_1b_{11}b_{12} &+& l_2b_{21}b_{22} &+\dots+& l_nb_{n1}b_{n2} =& 0, \\
l_1b_{11}b_{13} &+& l_2b_{21}b_{23} &+\dots+& l_nb_{n1}b_{n3} =& 0,  \\
\vdots & \vdots & \vdots & \vdots & \vdots \\
l_1b_{1(n-1)}b_{1n} &+& l_2b_{2(n-1)}b_{2n} &+\dots+&
l_nb_{n(n-1)}b_{nn} =& 0.
\end{array}
\right.
\end{align}
By Jacobi's triple product identity \eqref{JTI}, now we can rewrite
\eqref{case I} as a product of $n$ theta functions
\begin{align}
&f({a_1}^{\frac{b_{11}^2+b_{11}}{2}}{b_1}^{\frac{b_{11}^2-b_{11}}{2}}\cdots
{a_n}^{\frac{b_{n1}^2+b_{n1}}{2}}{b_n}^{\frac{b_{n1}^2-b_{n1}}{2}},
{a_1}^{\frac{b_{11}^2-b_{11}}{2}}{b_1}^{\frac{b_{11}^2+b_{11}}{2}}\cdots
{a_n}^{\frac{b_{n1}^2-b_{n1}}{2}}{b_n}^{\frac{b_{n1}^2+b_{n1}}{2}})
\cdots \notag \\
\label{case I1}&\times
f({a_1}^{\frac{b_{1n}^2+b_{1n}}{2}}{b_1}^{\frac{b_{1n}^2-b_{1n}}{2}}\cdots
{a_n}^{\frac{b_{nn}^2+b_{nn}}{2}}{b_n}^{\frac{b_{nn}^2-b_{nn}}{2}},
{a_1}^{\frac{b_{1n}^2-b_{1n}}{2}}{b_1}^{\frac{b_{1n}^2+b_{1n}}{2}}\cdots
{a_n}^{\frac{b_{nn}^2-b_{nn}}{2}}{b_n}^{\frac{b_{nn}^2+b_{nn}}{2}}).
\end{align}

General case: If $y\equiv c_r\pmod d$ $(r=0, 1,2, \dots, k-1)$ is a
solution of \eqref{congruence equation mod d}, we substitute $y$
with $dy+c_r$ in \eqref{x y}. Then we have $x=By+\frac{1}{d}Bc_r$,
where $\frac{1}{d}Bc_r$ is an $n$-dimensional integer vector. By
adding the contribution of each solution to the sum $S$, we can
write $S$ as a linear combination of products of theta functions. We
need \eqref{orthogonal} in all cases in order to separate $y_1, y_2,
\dots, y_n$.

We can find that in the first theta function in \eqref{case I1},
\begin{align*}
&{a_1}^{\frac{b_{11}^2+b_{11}}{2}}{b_1}^{\frac{b_{11}^2-b_{11}}{2}}\cdots
{a_n}^{\frac{b_{n1}^2+b_{n1}}{2}}{b_n}^{\frac{b_{n1}^2-b_{n1}}{2}}\times
{a_1}^{\frac{b_{11}^2-b_{11}}{2}}{b_1}^{\frac{b_{11}^2+b_{11}}{2}}\cdots
{a_n}^{\frac{b_{n1}^2-b_{n1}}{2}}{b_n}^{\frac{b_{n1}^2+b_{n1}}{2}}\notag
\\
=&{(a_1b_1)}^{{b_{11}}^2}\cdots
{(a_nb_n)}^{{b_{n1}}^2}=q^{l_1{b_{11}}^2+\cdots+l_n{b_{n1}}^2}.
\end{align*}
Similarly we can find that the rest of the products in \eqref{case
I1} are $q^{l_1{b_{12}}^2+\cdots+l_n{b_{n2}}^2}$, $\cdots$,
$q^{l_1{b_{1n}}^2+\cdots+l_n{b_{nn}}^2}$. It can be shown that this
pattern holds for any part in the linear combination. It is very
useful in finding the theta functions in the linear combination and
is illustrated later in the proof of \eqref{hahn1}.

Let $B=(\mathbf{b_1}, \dots, \mathbf{b_n})$, where $\mathbf{b_j}$ is
the $j$th column of $B$. If we choose integers $l_1=l_2=\cdots=l_n$,
then \eqref{orthogonal} implies that $\{\mathbf{b_1}, \dots,
\mathbf{b_n}\}$ is an orthogonal set. If the set of all the columns
of a matrix $B$ is an orthogonal set, then $B^TB$ is a diagonal
matrix with all entries positive integers on the main diagonal. We
define
\[D=\begin{pmatrix}

l_{1} & 0 & 0 & \cdots & 0\\

0 & l_{2} & 0 & \cdots & 0\\

0 & 0 & l_{3} & \cdots & 0 \\

\vdots & \vdots & \vdots & \ddots & \vdots \\

0 & 0 & 0 & \cdots & l_{n}
\end{pmatrix}\ ,\]
where $l_1,l_2,\dots, l_n$ are positive integers. Then
\eqref{orthogonal} implies that $B^TDB$ is a diagonal matrix with
all entries positive integers on the main diagonal. So the set of
all the columns of $B$ is a kind of ``generalized orthogonal'' set.
For fixed $l_i$ $(i=1, 2, \dots, n)$, there are infinitely many
solutions $b_{ij}$ $(i, j=1, 2, \dots, n)$ to the system of
equations \eqref{orthogonal}. So product of any $n$ theta functions
can be written as the linear combinations of other products of theta
functions, not in unique way.

 We have $\det A^*={(\det A)}^{n-1}$. Since $A\cdot
A^*=\det A\cdot I$, where $A$ is an integer matrix, we require
$(\hbox{det} A)\cdot {(A^*)}^{-1}$ to be an integer matrix.

Without lose of generality, we have the following procedure for
obtaining series-product identities.\\
1. For fixed positive integers $l_1,l_2, \dots, l_n$, find all
$n\times n$ matrices $B'$ satisfying the ``generalized orthogonal''
relation \eqref{orthogonal}, where $\det B'=\pm {(d')}^{n-1}$ and
$d'{B'}^{-1}$ is an integer matrix, where $d'\in N$.
\\
2. For the system of congruences $B'y\equiv 0\pmod {d'}$ obtained
from step 1, we divide by the greatest common divisor on both sides
and rewrite it as $By\equiv 0\pmod d$. Next we solve the system of
congruences $By\equiv 0\pmod d$. Suppose we have $k$ solutions. By
computing the contribution of each solution, we can write each
product of $n$ theta functions as the linear combination of $k$
products of $n$ theta functions.

If $y\equiv c_r\pmod d$ $(r=0, 1,2, \dots, k-1)$ is the solution set
of \eqref{congruence equation mod d}, we substitute each $y$ with
$dy+c_r$ in \eqref{x y}. Then we have $x=By+\frac{1}{d}Bc_r$, $(r=0,
1,2, \dots, k-1)$. We assume $c_0=0$. It is easy to see that for
$y\in \Z^n$, $\{By+\frac{1}{d}Bc_r\}^{k-1}_{r=0}$ covers $\Z^n$ and
there is no overlap between the members. By defining the binary
operation as vector addition in the abelian additive group $\Z^n$,
$B\Z^n \lhd \Z^n$, and two vectors $\mathbf{a} \equiv \mathbf{b}$ if
and only if $\mathbf{{a}}-\mathbf{b}\in B\Z^n$. So now we can take
$\{B\Z^n+\frac{1}{d}Bc_r\}^{k-1}_{r=0}$ as an integer matrix exact
covering system of $\Z^n$. In other words, a partition of $\Z^n$
into a lattice and a finite number of its translates. If we view
this geometrically, we are counting the distinct points inside the
$n$-dimensional parallelotope spanned by the columns of $B$ (Two
points are counted as one if the difference of the corresponding
vector is a linear combination of the columns of $B$). From M.~A.
Fiol \cite[Proposition 2]{fiol}, the number of equivalence classes
in the quotient group is $|\det B|$. So $k=|\det B|$. We know that
$|\det B| =$ volume of the parallelotope spanned by the columns of
$B$. So this implies that numerically the number of distinct points
inside a parallelotope equals the volume of the parallelotope.

So now it is clear that we are really looking for a special kind of
ECS of $\Z^n$ corresponds to ``generalized orthogonal'' matrix $B$.
Although we can find the ECS by solving the system of congruences
\eqref{congruence equation mod d}, there is no general formula for
the solution set of system of congruences. A better approach is to
find the congruence classes by finding $k$ distinct points in $\Z^n$
if $B$ is known with $|\det B|=k$. We have the same covering system
if we multiply a column of $B$ by $-1$, or we interchange two
columns of $B$. So here we can assume $\det B>0$.

By choosing special coset representatives, we give the main theorem
for products of $n$ theta functions.
\begin{thm} \label{main theorem}
Let $l_i \in \Z^{+}$, $a_ib_i=q^{l_i}$, ($i=1,2, \dots, n$). Let
$B=(b_{ij})_{n\times n}$ be an invertible integer matrix satisfying
\eqref{orthogonal}. Let $k=\det B>0$.  Let $B^*$ be the adjugate of
$B$. At least one of the entries of $B^*$ has no common factor with
$k$. We suppose an entry in the $j$th column of $B^*$ is coprime to
$k$. Then
\begin{align}
&\prod_{i=1}^nf(a_i, b_i)\notag \\
=&\sum^{k-1}_{r=0}{a_j}^{\frac{r^2+r}{2}}{b_j}^{\frac{r^2-r}{2}}f\Big({a_1}^{\frac{b_{11}^2+b_{11}}{2}}{b_1}^{\frac{b_{11}^2-b_{11}}{2}}
\cdots{a_j}^{\frac{b_{j1}^2+b_{j1}}{2}+b_{j1}r}{b_j}^{\frac{b_{j1}^2-b_{j1}}{2}+b_{j1}r}\cdots
{a_n}^{\frac{b_{n1}^2+b_{n1}}{2}}{b_n}^{\frac{b_{n1}^2-b_{n1}}{2}},\notag
\\
&{a_1}^{\frac{b_{11}^2-b_{11}}{2}}{b_1}^{\frac{b_{11}^2+b_{11}}{2}}
\cdots{a_j}^{\frac{b_{j1}^2-b_{j1}}{2}-b_{j1}r}{b_j}^{\frac{b_{j1}^2+b_{j1}}{2}-b_{j1}r}\cdots
{a_n}^{\frac{b_{n1}^2-b_{n1}}{2}}{b_n}^{\frac{b_{n1}^2+b_{n1}}{2}}\Big)\cdots \notag \\
\label{genthm} &\times
f\Big({a_1}^{\frac{b_{1n}^2+b_{1n}}{2}}{b_1}^{\frac{b_{1n}^2-b_{1n}}{2}}
\cdots{a_j}^{\frac{b_{jn}^2+b_{jn}}{2}+b_{jn}r}{b_j}^{\frac{b_{jn}^2-b_{jn}}{2}+b_{jn}r}\cdots
{a_n}^{\frac{b_{nn}^2+b_{nn}}{2}}{b_n}^{\frac{b_{nn}^2-b_{nn}}{2}},\notag
\\
&{a_1}^{\frac{b_{1n}^2-b_{1n}}{2}}{b_1}^{\frac{b_{1n}^2+b_{1n}}{2}}
\cdots{a_j}^{\frac{b_{jn}^2-b_{jn}}{2}-b_{jn}r}{b_j}^{\frac{b_{jn}^2+b_{jn}}{2}-b_{jn}r}\cdots
{a_n}^{\frac{b_{nn}^2-b_{nn}}{2}}{b_n}^{\frac{b_{nn}^2+b_{nn}}{2}}\Big).
\end{align}
\end{thm}
\begin{rem}
The above theorem remains valid if the sum over $r$ runs over a
complete residue system $\pmod k$. For instance, $[-\frac{k}{2}]+1
\leq r \leq [\frac{k}{2}]$.
\end{rem}
\begin{proof}
For \eqref{genthm}, we only need to show that for integer matrix
exact covering system \[\Bigg\{B\Z^n+\left(
\begin{array}{rrrr}
0\\
\vdots\\
r\\
\vdots\\
0
\end{array}
\right)\Bigg\}_{r=0}^{k-1}\,,\] where $r$ is in the $j$th row, there
is no overlap between the members of the covering system. That is,
$By=\left(
\begin{array}{rrrr}
0\\
\vdots\\
r\\
\vdots\\
0
\end{array}
\right),$ $(r=1, 2, \dots, k-1)$ has no integer vector solutions.

From $By=\left(
\begin{array}{rrrr}
0\\
\vdots\\
r\\
\vdots\\
0
\end{array}
\right)$, by the inverse formula, we have \[y=\frac{1}{k}B^*\left(
\begin{array}{rrrr}
0\\
\vdots\\
r\\
\vdots\\
0
\end{array}
\right)=\frac{1}{k}\left(
\begin{array}{rrrr}
b^*_{1j}r\\
b^*_{2j}r\\
\vdots\\
b^*_{nj}r
\end{array}
\right).\] Since an entry in the $j$th column of $B^*$ is coprime to
$k$, at least one entry of $\left(
\begin{array}{rrrr}
b^*_{1j}r\\
b^*_{2j}r\\
\vdots\\
b^*_{nj}r
\end{array}
\right)$ is not divisible by $k$.  So $By=\left(
\begin{array}{rrrr}
0\\
\vdots\\
r\\
\vdots\\
0
\end{array}
\right), $ $(r=1, 2, \dots, k-1)$ has no integer vector solution. We
finish the proof.
\end{proof}
Theorem \ref{genthm} is symmetric with respect to $a_i$ and $b_i$
$(i=1, 2, \cdots, n)$. Interchanging the positions of $a_i$ and
$b_i$ is equivalent to multiple by $-1$ on the $i$th row of $B$.

\begin{cor}
In Theorem \ref{2 main theorem}, if we further require that for any
column of $B$, the sum of all the entries is an even number, then
\begin{align}
&\prod_{i=1}^nf(a_i, b_i)+\prod_{i=1}^nf(-a_i, -b_i)\notag \\
=&2\sum^{\frac{k-2}{2}}_{r=0}{a_j}^{2r^2+r}{b_j}^{2r^2-r}
f\Big({a_1}^{\frac{b_{11}^2+b_{11}}{2}}{b_1}^{\frac{b_{11}^2-b_{11}}{2}}
\cdots{a_j}^{\frac{b_{j1}^2+b_{j1}}{2}+2b_{j1}r}{b_j}^{\frac{b_{j1}^2-b_{j1}}{2}+2b_{j1}r}\cdots
{a_n}^{\frac{b_{n1}^2+b_{n1}}{2}}{b_n}^{\frac{b_{n1}^2-b_{n1}}{2}},\notag
\\
&{a_1}^{\frac{b_{11}^2-b_{11}}{2}}{b_1}^{\frac{b_{11}^2+b_{11}}{2}}
\cdots{a_j}^{\frac{b_{j1}^2-b_{j1}}{2}-2b_{j1}r}{b_j}^{\frac{b_{j1}^2+b_{j1}}{2}-2b_{j1}r}\cdots
{a_n}^{\frac{b_{n1}^2-b_{n1}}{2}}{b_n}^{\frac{b_{n1}^2+b_{n1}}{2}}\Big)\cdots \notag \\
&\times
f\Big({a_1}^{\frac{b_{1n}^2+b_{1n}}{2}}{b_1}^{\frac{b_{1n}^2-b_{1n}}{2}}
\cdots{a_j}^{\frac{b_{jn}^2+b_{jn}}{2}+2b_{jn}r}{b_j}^{\frac{b_{jn}^2-b_{jn}}{2}+2b_{jn}r}\cdots
{a_n}^{\frac{b_{nn}^2+b_{nn}}{2}}{b_n}^{\frac{b_{nn}^2-b_{nn}}{2}},\notag
\\
\label{genthmsum1}
&{a_1}^{\frac{b_{1n}^2-b_{1n}}{2}}{b_1}^{\frac{b_{1n}^2+b_{1n}}{2}}
\cdots{a_j}^{\frac{b_{jn}^2-b_{jn}}{2}-2b_{jn}r}{b_j}^{\frac{b_{jn}^2+b_{jn}}{2}-2b_{jn}r}\cdots
{a_n}^{\frac{b_{nn}^2-b_{nn}}{2}}{b_n}^{\frac{b_{nn}^2+b_{nn}}{2}}\Big).
\end{align}

\begin{align}
&\prod_{i=1}^nf(a_i, b_i)-\prod_{i=1}^nf(-a_i,
-b_i)\notag \\
=&2\sum^{\frac{k-2}{2}}_{r=0}{a_j}^{2r^2+3r+1}{b_j}^{2r^2+r}\notag \\
&f({a_1}^{\frac{b_{11}^2+b_{11}}{2}}{b_1}^{\frac{b_{11}^2-b_{11}}{2}}
\cdots{a_j}^{\frac{b_{j1}^2+3b_{j1}}{2}+2b_{j1}r}{b_j}^{\frac{b_{j1}^2+b_{j1}}{2}+2b_{j1}r}\cdots
{a_n}^{\frac{b_{n1}^2+b_{n1}}{2}}{b_n}^{\frac{b_{n1}^2-b_{n1}}{2}},\notag
\\
&{a_1}^{\frac{b_{11}^2-b_{11}}{2}}{b_1}^{\frac{b_{11}^2+b_{11}}{2}}
\cdots{a_j}^{\frac{b_{j1}^2-3b_{j1}}{2}-2b_{j1}r}{b_j}^{\frac{b_{j1}^2-b_{j1}}{2}-2b_{j1}r}\cdots
{a_n}^{\frac{b_{n1}^2-b_{n1}}{2}}{b_n}^{\frac{b_{n1}^2+b_{n1}}{2}})\cdots \notag \\
&\times
f\Big({a_1}^{\frac{b_{1n}^2+b_{1n}}{2}}{b_1}^{\frac{b_{1n}^2-b_{1n}}{2}}
\cdots{a_j}^{\frac{b_{jn}^2+3b_{jn}}{2}+2b_{jn}r}{b_j}^{\frac{b_{jn}^2+b_{jn}}{2}+2b_{jn}r}\cdots
{a_n}^{\frac{b_{nn}^2+b_{nn}}{2}}{b_n}^{\frac{b_{nn}^2-b_{nn}}{2}},\notag
\\
\label{genthmsum2}
&{a_1}^{\frac{b_{1n}^2-b_{1n}}{2}}{b_1}^{\frac{b_{1n}^2+b_{1n}}{2}}
\cdots{a_j}^{\frac{b_{jn}^2-3b_{jn}}{2}-2b_{jn}r}{b_j}^{\frac{b_{jn}^2-b_{jn}}{2}-2b_{jn}r}\cdots
{a_n}^{\frac{b_{nn}^2-b_{nn}}{2}}{b_n}^{\frac{b_{nn}^2+b_{nn}}{2}}\Big).
\end{align}
\end{cor}
\begin{proof}
Since for any column of $B$, the sum of all the entries is an even
number, it is easy to see that $k$ is even, and all of
${{b_{1j}}^2+\cdots {b_{nj}}^2}$ $(j=1, 2, \cdots, n)$ are even.
Replacing $a_i$ with $-a_i$, $b_i$ with $-b_i$ in \eqref{genthm},
adding it to \eqref{genthm}, and finally replacing $2r$ with $r$, we
obtain \eqref{genthmsum1}. The proof of \eqref{genthmsum2} is
similar.
\end{proof}

 Next we give an estimate of $\det B$
if all the entries of $B$ are nonzero. We have
\[B^TDB=\]
\begin{align} \label{matrix}
\begin{pmatrix}
l_1b_{11}^2+l_2b_{21}^2+\cdots+l_nb_{n1}^2 & 0 & 0 & \cdots & 0\\
0 & l_1b_{12}^2+l_2b_{22}^2+\cdots+l_nb_{n2}^2 & 0 & \cdots & 0\\
\vdots & \vdots & \vdots & \ddots & \vdots \\
0 & 0 & 0 & \cdots & l_1b_{1n}^2+l_2b_{2n}^2+\cdots+l_nb_{nn}^2
\end{pmatrix}.
\end{align}
We take the determinant of both sides of \eqref{matrix} if all the
entries of $B$ are nonzero and obtain
\begin{align*}
&{(\hbox{det} B)}^2\cdot l_1l_2\cdots l_n
\notag \\
=&(l_1b_{11}^2+l_2b_{21}^2+\cdots+l_nb_{n1}^2)(l_1b_{12}^2+l_2b_{22}^2+\cdots+l_2b_{n2}^2)\cdots(l_1b_{1n}^2+l_2b_{2n}^2+\cdots+l_nb_{nn}^2).
\end{align*}
So
\begin{align} \label{detb}
\hbox{det}
B=\sqrt{{\frac{(l_1b_{11}^2+l_2b_{21}^2+\cdots+l_nb_{n1}^2)\cdots(l_1b_{1n}^2+l_2b_{2n}^2+\cdots+l_nb_{nn}^2)}{l_1l_2
\cdots l_n}}}.
\end{align}

We have \begin{align} \label{det2}
l_1b_{1i}^2+l_2b_{2i}^2+\cdots+l_nb_{ni}^2\geq
n{(l_1b_{1i}^2l_2b_{2i}^2\cdots l_nb_{ni}^2)}^\frac{1}{n}\quad (i=1,
2, \dots, n).
\end{align}
By \eqref{detb} and \eqref{det2}, if we denote the product of all
entries of the matrix $B$ by $\triangle B$, then we can obtain
\begin{align*}
\det B \geq n^{\frac{n}{2}}(\triangle B)^{\frac{1}{n}}.
\end{align*}

In order to minimize the number of solutions of the system of
congruences \eqref{congruence equation mod d}, we need to find
``generalized orthogonal" matrices with absolute values of
determinants as small as possible. All the 3 by 3 and 4 by 4
matrices in this paper have entries chosen from $-2, -1, 0, 1, 2$.

\section{Products of Two Theta Functions}
For a product of two theta functions $f(a, b)f(c,d)$ with
$ab=q^{l_1}$ and $cd=q^{l_2}$, we need to find $2$ by $2$ invertible
integer matrices $B$ satisfying the generalized orthogonal relation
$l_1b_{11}b_{12}+l_2b_{21}b_{22}=0$. Without lose of generality, we
can assume that three of $b_{11}, b_{12}, b_{21}, b_{22}$ are
positive, one is negative. Now we give the following theorem for
product of two theta functions, which is a special case of Theorem
\ref{main theorem}.
\begin{thm} \label{2 main theorem}
Let $ab=q^{l_1}$ and $cd=q^{l_2}$, $l_i \in \Z^{+}$, where $i=1$,
$2$, and let $B=(b_{ij})$ be a $2\times2$ invertible integer matrix
such that $l_1b_{11}b_{12}+l_2b_{21}b_{22}=0$, $\gcd(b_{21},
b_{11}b_{22})=1$ or $\gcd(b_{22}, b_{12}b_{21})=1$. Let $k=\det
B=b_{11}b_{22}-b_{12}b_{21}> 0$. Then
\begin{align}
&f(a, b)f(c,d)\notag \\
=&\sum^{k-1}_{r=0}{a}^{\frac{r^2+r}{2}}{b}^{\frac{r^2-r}{2}}f(a^{\frac{b_{11}^2+b_{11}}{2}+b_{11}r}b^{\frac{b_{11}^2-b_{11}}{2}+b_{11}r}
c^{\frac{b_{21}^2+b_{21}}{2}}d^{\frac{b_{21}^2-b_{21}}{2}},
a^{\frac{b_{11}^2-b_{11}}{2}-b_{11}r}b^{\frac{b_{11}^2+b_{11}}{2}-b_{11}r}
c^{\frac{b_{21}^2-b_{21}}{2}}d^{\frac{b_{21}^2+b_{21}}{2}})\notag \\
\label{2genthmc} &\times
f(a^{\frac{b_{12}^2+b_{12}}{2}+b_{12}r}b^{\frac{b_{12}^2-b_{12}}{2}+b_{12}r}
c^{\frac{b_{22}^2+b_{22}}{2}}d^{\frac{b_{22}^2-b_{22}}{2}},
a^{\frac{b_{12}^2-b_{12}}{2}-b_{12}r}b^{\frac{b_{12}^2+b_{12}}{2}-b_{12}r}
c^{\frac{b_{22}^2-b_{22}}{2}}d^{\frac{b_{22}^2+b_{22}}{2}}).
\end{align}
\end{thm}
\begin{proof}
If $\gcd(b_{21}, b_{11}b_{22})=1$ or $\gcd(b_{22}, b_{12}b_{21})=1$,
then either $\gcd(b^*_{11}, k)=1$ or $\gcd(b^*_{21}, k)=1$. The
conditions of Theorem \ref{main theorem} are met.
\end{proof}
If $\gcd(b_{11}, b_{12}b_{21})=1$ or $\gcd(b_{12}, b_{11}b_{22})=1$,
then we have the following identity corresponding to
\eqref{2genthmc}
\begin{align}
&f(a, b)f(c,d)\notag \\
=&\sum^{k-1}_{r=0}{c}^{\frac{r^2+r}{2}}{d}^{\frac{r^2-r}{2}}f(a^{\frac{b_{11}^2+b_{11}}{2}}b^{\frac{b_{11}^2-b_{11}}{2}}
c^{\frac{b_{21}^2+b_{21}}{2}+b_{21}r}d^{\frac{b_{21}^2-b_{21}}{2}+b_{21}r},
a^{\frac{b_{11}^2-b_{11}}{2}}b^{\frac{b_{11}^2+b_{11}}{2}}
c^{\frac{b_{21}^2-b_{21}}{2}-b_{21}r}d^{\frac{b_{21}^2+b_{21}}{2}-b_{21}r})\notag \\
\label{2genthm} &\times
f(a^{\frac{b_{12}^2+b_{12}}{2}}b^{\frac{b_{12}^2-b_{12}}{2}}
c^{\frac{b_{22}^2+b_{22}}{2}+b_{22}r}d^{\frac{b_{22}^2-b_{22}}{2}+b_{22}r},
a^{\frac{b_{12}^2-b_{12}}{2}}b^{\frac{b_{12}^2+b_{12}}{2}}
c^{\frac{b_{22}^2-b_{22}}{2}-b_{22}r}d^{\frac{b_{22}^2+b_{22}}{2}-b_{22}r}).
\end{align}

If $b_{11}$ and $b_{21}$, $b_{12}$ and $b_{22}$ are of the same
parity, then we have the following corollary.
\begin{cor}
In Theorem \ref{2 main theorem}, if we further require that $b_{11}$
and $b_{21}$, $b_{12}$ and $b_{22}$ are of the same parity, then
\begin{align}
&f(a, b)f(c,d)+f(-a, -b)f(-c, -d)\\
=&2\sum^{\frac{k-2}{2}}_{r=0}{a}^{2r^2+r}{b}^{2r^2-r}\notag \\
&
f\Big(a^{\frac{b_{11}^2+b_{11}}{2}+2b_{11}r}b^{\frac{b_{11}^2-b_{11}}{2}+2b_{11}r}
c^{\frac{b_{21}^2+b_{21}}{2}}d^{\frac{b_{21}^2-b_{21}}{2}},
a^{\frac{b_{11}^2-b_{11}}{2}-2b_{11}r}b^{\frac{b_{11}^2+b_{11}}{2}-2b_{11}r}
c^{\frac{b_{21}^2-b_{21}}{2}}d^{\frac{b_{21}^2+b_{21}}{2}}\Big)\notag \\
\label{2genthmcc} &\times
f\Big(a^{\frac{b_{12}^2+b_{12}}{2}+2b_{12}r}b^{\frac{b_{12}^2-b_{12}}{2}+2b_{12}r}
c^{\frac{b_{22}^2+b_{22}}{2}}d^{\frac{b_{22}^2-b_{22}}{2}},
a^{\frac{b_{12}^2-b_{12}}{2}-2b_{12}r}b^{\frac{b_{12}^2+b_{12}}{2}-2b_{12}r}
c^{\frac{b_{22}^2-b_{22}}{2}}d^{\frac{b_{22}^2+b_{22}}{2}}\Big).
\end{align}
\begin{align}
&f(a, b)f(c,d)-f(-a, -b)f(-c, -d)\notag \\
=2&\sum^{\frac{k-2}{2}}_{r=0}{a}^{2r^2+3r+1}{b}^{2r^2+r}\notag \\
&f\Big(a^{\frac{b_{11}^2+3b_{11}}{2}+2b_{11}r}b^{\frac{b_{11}^2+b_{11}}{2}+2b_{11}r}
c^{\frac{b_{21}^2+b_{21}}{2}}d^{\frac{b_{21}^2-b_{21}}{2}},
a^{\frac{b_{11}^2-3b_{11}}{2}-2b_{11}r}b^{\frac{b_{11}^2-b_{11}}{2}-2b_{11}r}
c^{\frac{b_{21}^2-b_{21}}{2}}d^{\frac{b_{21}^2+b_{21}}{2}}\Big)\notag \\
\label{2genthmccc} &\times
f\Big(a^{\frac{b_{12}^2+3b_{12}}{2}+2b_{12}r}b^{\frac{b_{12}^2+b_{12}}{2}+2b_{12}r}
c^{\frac{b_{22}^2+b_{22}}{2}}d^{\frac{b_{22}^2-b_{22}}{2}},
a^{\frac{b_{12}^2-3b_{12}}{2}-2b_{12}r}b^{\frac{b_{12}^2-b_{12}}{2}-2b_{12}r}
c^{\frac{b_{22}^2-b_{22}}{2}}d^{\frac{b_{22}^2+b_{22}}{2}}\Big).
\end{align}
\end{cor}
Without lose of generality, it is easy to see that the simplest
nontrivial 2 by 2 ``generalized orthogonal" matrix is
 \[ B=\left(
\begin{array}{rr}
1 & 1\\
-1 & 1
\end{array}
\right),
\]
with $\det B=2$. So $k \geqslant 2$ for 2 by 2 matrices throughout
the remainder of this paper.
\begin{cor} \label{k1k2cor}
If $|ab|<1$ and $(cd)=(ab)^{k_1k_2}$, where both $k_1$ and $k_2$ are
positive integers, then
\begin{align}
f(a,b)f(c,d)=&\sum_{r=0}^{k_{1}+k_{2}-1}a^{\frac{r^2+r}{2}}b^{\frac{r^2-r}{2}}f\Big(a^{\frac{k_1^2+k_1}{2}+k_1r}b^{\frac{k_1^2-k_1}{2}+k_1r}d,
a^{\frac{k_1^2-k_1}{2}-k_1r}b^{\frac{k_1^2+k_1}{2}-k_1r}c\Big)\notag \\
\label{k1k2} &\times
f\Big(a^{\frac{k_2^2+k_2}{2}+k_2r}b^{\frac{k_2^2-k_2}{2}+k_2r}c,
a^{\frac{k_2^2-k_2}{2}-k_2r}b^{\frac{k_2^2+k_2}{2}-k_2r}d\Big).
\end{align}
\end{cor}
\noindent It is easy too see that \eqref{k1k2} is symmetric with
respect to $a$ and $b$, $c$ and $d$, and $k_1$ and $k_2$.
\begin{proof} Choose
\[
B=\left(
\begin{array}{rr}
k_1 & k_2\\
-1 & 1
\end{array}
\right)
\]
in  \eqref{2genthmc}.
\end{proof}
Many identities involving products of two theta functions are
special cases of Corollary \ref{k1k2cor}. We give a list of
identities here.
\begin{cor}
\begin{align}
\label{phiphi}
\varphi(q)\varphi(q^m)&=\sum_{r=0}^{m}q^{r^2}f(q^{m+1+2r},
q^{m+1-2r})f(q^{m^2+m+2mr},
q^{m^2+m-2mr}), \\
\label{phiphi2}
\varphi(q)\varphi(q^{m^2})&=\sum_{r=0}^{2m-1}q^{r^2}f^2(q^{2m^2+2mr},
q^{2m^2-2mr}),\\
\psi(q)\varphi(q^{2m^2})&=\sum_{r=0}^{2m-1}q^{2r^2-r}f^2(q^{4m^2-m+4mr},
q^{4m^2+m-4mr}),\\
2\psi(q)\psi(q^{4m^2})&=\sum_{r=0}^{2m-1}q^{2r^2-r}f(q^{6m^2-m+4mr},
q^{2m^2+m-4mr})f(q^{2m^2-m+4mr}, q^{6m^2+m-4mr}),\\
2\psi(q)\varphi(q^m)&=\sum_{r=0}^{2m}q^{\frac{r^2-r}{2}}f(q^{m+r},
q^{m+1-r})f(q^{2m^2+mr}, q^{2m^2+2m-mr}),\\
\psi(q)\varphi(q^{2m})&=\sum_{r=0}^{m}q^{2r^2-r}f(q^{2m+1+4r},
q^{2m+3-4r})f(q^{2m^2+m+4mr}, q^{2m^2+3m-4mr}),\\
2\varphi(q)\psi(q^{2m})&=\sum_{r=0}^{m}q^{r^2}f(q^{2m+1+2r},
q^{1-2r})f(q^{m^2+2mr}, q^{m^2+2m-2mr}),\\
\psi(q)\psi(q^m)&=\sum_{r=0}^{m}q^{2r^2-r}f(q^{3m+1+4r},
q^{m+3-4r})f(q^{2m^2+4mr}, q^{2m^2+4m-4mr}).
\end{align}
\end{cor}
If $k_1k_2$ is a composite number, then the right-hand side of
\eqref{k1k2} is not unique, for it depends on which $k_1$ and $k_2$
we choose. For example, if we choose $m=4$ in \eqref{phiphi}, we
have
\[\varphi(q)\varphi(q^4)=\varphi(q^5)\phi(q^{20})+2qf(q^3,
q^7)f(q^{12}, q^{28})+2q^4f(q, q^9)f(q^{4}, q^{36}).\]
If we choose
$m=2$ in \eqref{phiphi2}, we have a different representation of
$\varphi(q)\varphi(q^4)$
\begin{align*}
\varphi(q)\varphi(q^4)=\varphi^2(q^8)+2q\psi^2(q^4)+4q^4\psi^2(q^{16}).
\end{align*}

The G\"{o}llnitz-Gordon functions are defined by
\begin{align*}
S(q):=&\sum_{n=0}^{\infty}\frac{(-q; q^2)_n}{(q^2;
q^2)_n}q^{n^2}=\frac{1}{(q; q^8)_\infty(q^4; q^8)_\infty(q^7;
q^8)_\infty},\notag \\
T(q):=&\sum_{n=0}^{\infty}\frac{(-q; q^2)_n}{(q^2;
q^2)_n}q^{n^2+n}=\frac{1}{(q^3; q^8)_\infty(q^4; q^8)_\infty(q^5;
q^8)_\infty}.
\end{align*}

Huang \cite{huang3}, S. L. Chen and S. S. Huang \cite {huang2} found
a list of modular relations for the G\"{o}llnitz-Gordon functions
which are analogous to Ramanujan's forty identities for the
Rogers-Ramanujan functions. Some of their identities can be shown as
special cases of our Theorem \ref{2 main theorem} and we give proofs
to four of their identities in this paper.
\begin{cor}
Let $f_n=f(-q^n)$ for nonnegative integer $n$.
\begin{align}
\label{huangb}
S(q^7)T(q)-q^3S(q)T(q^7)&=1,\\
\label{huangb2}
S(q^3)S(q)+q^2T(q^3)T(q)&=\frac{f_3f_4}{f_1f_{12}},\\
\label{huangb3}
S(q)S(q)+qT(q)T(q)&=\frac{f^6_2}{f^3_1f^3_{4}},\\
\label{huangb4} S(q^3)T(q)-qS(q)T(q^3)&=\frac{f_1f_{12}}{f_3f_{4}}.
\end{align}
\end{cor}

Of Huang's identities, the most beautiful one is \eqref{huangb}. It is an analogue of Ramanujan's ``most beautiful'' identity of the forty identities
\[H(q)G(q^{11}) - q^2G(q)H(q^{11}) = 1,  \]
where
\[G(q): =\sum_{n=0}^{\infty}\frac{q^{n^2}}{(q; q)_n}=\frac{1}{(q; q^5)_\infty(q^4; q^5)_\infty}\quad \hbox{and} \quad H(q): =\sum_{n=0}^{\infty}\frac{q^{n(n+1)}}{(q; q)_n}=\frac{1}{(q^2; q^5)_\infty(q^3; q^5)_\infty}.\]
\begin{proof}If we choose $a=1$, $b=q$, $c=1$, $d=q^{m}$, $k_1=1$ and $k_2=m$ in
\eqref{k1k2}, then
\begin{align} \label{4psipsi}
f(1,q)f(1,
q^m)=4\psi(q)\psi(q^m)=\sum_{r=0}^{m}q^{\frac{r^2-r}{2}}f(q^{m+r},
q^{1-r})f(q^{\frac{m^2-m}{2}+mr}, q^{\frac{m^2+3m}{2}-mr}).
\end{align}
We let $m=7$ in \eqref{4psipsi}, then apply \eqref{aa} to find that
\begin{align}
2\psi(q)\psi(q^7)=&f(q, q^7)f(q^{21}, q^{35})+q^3f(q^{3},
q^{5})f(q^{7}, q^{49})\notag \\
\label{psiqpsiq71}
&+q\psi(q^2)\psi(q^{14})+\psi(q^8)\varphi(q^{28})+q^{6}\varphi(q^4)\psi(q^{56}).
\end{align}

If we choose $a=-1$, $b=-q$, $c=-1$, $d=-q^{m}$, $k_1=1$ and $k_2=m$
in \eqref{k1k2}, by \eqref{f-1a}, $f(-1, -q)=0$. We obtain
\begin{align} \label{psipsi0}
\sum_{r=0}^{m}(-1)^rq^{\frac{r^2-r}{2}}f(q^{m+r},
q^{1-r})f(q^{\frac{m^2-m}{2}+mr}, q^{\frac{m^2+3m}{2}-mr})=0.
\end{align}
\noindent Letting $m=7$ in \eqref{psipsi0}, we have
\begin{align*}
\sum_{i=0}^{7}(-1)^rq^{\frac{r^2-r}{2}}f(q^{7+r},
q^{1-r})f(q^{21+7r}, q^{35-7r})=0.
\end{align*}
We can deduce
\begin{align} \label{psiqpsiq72}
f(q, q^7)f(q^{21}, q^{35})+q^3f(q^{3}, q^{5})f(q^{7},
q^{49})=\psi(q^8)\varphi(q^{28})+q^{6}\varphi(q^4)\psi(q^{56})+q\psi(q^2)\psi(q^{14})
\end{align}
after simplification.

From \eqref{psiqpsiq71} and \eqref{psiqpsiq72}, we obtain
\cite[p.~315, (19.1)]{Ber1}
\begin{align} \label{deg7}
\psi(q)\psi(q^7)=\psi(q^8)\varphi(q^{28})+q^{6}\varphi(q^4)\psi(q^{56})+q\psi(q^2)\psi(q^{14}),
\end{align}
and
\begin{align} \label{huangid1}
\psi(q)\psi(q^7)=f(q, q^7)f(q^{21}, q^{35})+q^3f(q^{3},
q^{5})f(q^{7}, q^{49}).
\end{align}

Identity \eqref{deg7} is equivalent to a modular equation of degree
7, namely,
\[(\alpha\beta)^{1/8}+\{(1-\alpha)(1-\beta)\}^{1/8}=1,\]
where $\beta$ has degree 7 over $\alpha$.

Identity \eqref{huangid1} is equivalent to \eqref{huangb}. This
proof of \eqref{huangid1} is simple compared with its proof in
\cite{huang2}.

Next we sketch the proofs of
\begin{align}
 \label{huang2}
\psi(q^2)\varphi(-q^3)&=f(-q^3, -q^5)f(-q^9, -q^{15})+q^2f(-q,
-q^7)f(-q^3, -q^{21}),\\
\label{huang3} \psi(q)\varphi(-q^2)&=f^2(-q^3, -q^5)+qf^2(-q,
-q^7),\\
\label{huang4} \psi(q^6)\varphi(q)&=f(q, q^7)f(q^9, q^{15})+qf(q^3,
q^5)f(q^3, q^{21}),
\end{align}
which are equivalent to \eqref{huangb2}, \eqref{huangb3}, and
\eqref{huangb4} respectively. In \eqref{k1k2}, choose $a=1$,
$b=q^2$, $c=d=-q^3$, $k_1=1$ and $k_3=3$. We obtain \eqref{huang2}
after simplification. In \eqref{k1k2}, choose $a=q$, $b=q^3$,
$c=d=-q^2$, and $k_1=k_2=1$ to obtain \eqref{huang3}. In
\eqref{k1k2}, choose $a=q$, $b=q$, $c=1$, $d=q^6$, $k_1=1$, $k_2=3$
and obtain \eqref{huang4}.
\end{proof}
In a manner similar to the derivations of \eqref{deg7} and
\eqref{huangid1}, we can obtain a list of identities for products of
$\psi$ functions after simplification.
\begin{cor}
\begin{align*}
\psi(q)\psi(q^3)&=\varphi(q^6)\psi(q^4)+q\varphi(q^2)\psi(q^{12}),
\notag \\
\psi(q)\psi(q^5)&=f(q, q^5)f(q^{10},
q^{20})+q^3\varphi(q^3)\psi(q^{30}),\notag \\
2\psi(q^2)\psi(q^3)&=f(q^2, q^3)f(q^9,
q^{21})+\psi(q^5)\varphi(q^{15})+q^2f(q, q^4)f(q^3, q^{27}),
\notag \\
\psi(q^3)\psi(q^5)&=f(q^3, q^5)f(q^{45}, q^{75})+q^8f(q,
q^7)f(q^{15}, q^{105}),
\notag \\
\psi(q^3)\psi(q^5)&=\psi(q^8)\varphi(
q^{60})+q^3\psi(q^2)\psi(q^{30})+q^{14}\varphi(q^4)\psi(q^{120}),
\notag \\
\psi(q^3)\psi(q^7)&=\psi(q^{10})\varphi( q^{105})+q^3f(q^4,
q^6)f(q^{63}, q^{147})+q^{16}f(q^2, q^8)f(q^{21}, q^{189}),
\notag \\
\psi(q^3)\psi(q^7)&=f(q^3, q^7)f(q^{84}, q^{126})+q^9f(q,
q^9)f(q^{42}, q^{168})+q^{25}\varphi(q^{5})\psi(q^{210}),
\notag \\
\psi(q)\psi(q^9)&=f(q, q^9)f(q^{36}, q^{54})+q^3f(q^3, q^7)f(q^{18},
q^{72})+q^{10}\varphi(q^{5})\psi(q^{90}),
\notag \\
\psi(q)\psi(q^9)&=\psi(q^{10})\varphi(q^{45})+qf(q^2, q^8)f(q^{27},
q^{63})+q^6f(q^4, q^6)f(q^{9}, q^{81}).
\end{align*}
\end{cor}

If we choose $k_1=k_2=1$ in \eqref{k1k2}, then we have the following
corollary.
\begin{cor}
For $ab=cd$,
\begin{align} \label{abcd1}
f(a,b)f(c,d)=f(ad,bc)f(ac, bd)+af({c}/{a}, a^2bd)f({d}/{a}, a^2bc).
\end{align}
\end{cor}
So for $ab=cd$, we have
\[f(a,b)f(c,d)+f(-a,-b)f(-c,-d)=2f(ad,bc)f(ac, bd).\] This is (i) of
Entry 29 in \cite[p.~45]{Ber1}. Similarly we can deduce (ii) of
Entry 29, which is
\[f(a,b)f(c,d)-f(-a,-b)f(-c,-d)=2af\bigg(\frac{b}{c},
\frac{c}{b}abcd\bigg)f\bigg(\frac{b}{d}, \frac{d}{b}abcd\bigg)\] for
$ab=cd$.

We choose $a=b=c=d=q$ in \eqref{abcd1}, then we have
\begin{align} \label{phi 1}
\varphi^2(q)=\varphi^2(q^2)+4q\psi^2(q^4).
\end{align}
Replacing $q$ with $q^2$ in \eqref {phi 1}, we can obtain
\begin{align*}
\varphi^2(q)=\varphi^2(q^4)+4q\psi^2(q^4)+4q^2\psi^2(q^8).
\end{align*}
Repeating this process, we have the following formula. For each
positive integer $n$,
\begin{align} \label{phi 2g}
\varphi^2(q)=\varphi^2(q^{2^n})+4q\psi^2(q^4)+\dots+4q^{2^{n-1}}\psi^2(q^{2^{n+1}}).
\end{align}

If we take the square of \eqref{phi 2g}, we have the following
formula. For each integer $n\geq 2$,
\begin{align}
{\varphi^4}(q)=\varphi^4(q^{2^n})+8q\psi^4(q^2)+\sum^{n-1}_{k=1}24q^{2^{k}}\psi^2(q^{2^{k+1}})+16q^{2^{n}}\psi^2(q^{2^{n+1}}).
\end{align}

Using \eqref{abcd1}, we can also prove Entry 10 in
\cite[p.~147]{Ber2}, namely
\begin{align*}
\big[f(a,b)-f(a^3b, ab^3)\big]^2=f(a^2, b^2)\varphi(ab)-f^2(a^3b,
ab^3).
\end{align*}
\begin{proof}
From \eqref{entry 311}, we have \[f(a,b)-f(a^3b, ab^3)=af({b}/{a},
a^5b^3).\] So we only need to show
\[f(a^2, b^2)\varphi(ab)=a^2f^2({b}/{a},
a^5b^3)+f^2(a^2b, ab^2),\] which can be deduced directly from
\eqref{abcd1}.
\end{proof}

J. A. Ewell's sextuple product identity in \cite{ewell2} is a
special case of \eqref{abcd1}.
\begin{cor} [Sextuple Product Identity] \label{ewell}
\begin{align*}
&(-xyq;q^2)_\infty(-q/xy; q^2)_\infty(-qx/y; q^2)_\infty(-qy/x;
q^2)_\infty(q^2; q^2)^2_\infty\notag \\
=&\sum^{\infty}_{i=-\infty}q^{2i^2}x^{2i}\sum^{\infty}_{j=-\infty}q^{2j^2}y^{2j}+q\sum^{\infty}_{i=-\infty}q^{2i(i+1)}x^{2i+1}\sum^{\infty}_{j=-\infty}q^{2j(j+1)}y^{2j+1}.
\end{align*}
\end{cor}
\begin{proof}
In \eqref{abcd1}, replacing $a$ with $xyq$, $b$ with $q/xy$, $c$
with $qx/y$, and $d$ with $qy/x$, and then applying Jacobi's triple
product identity, we complete the proof after simplification.
\end{proof}
Using \eqref{abcd1}, we can also prove the ``three term relation for
Weierstrass sigma functions'' given in \cite[p.~52, Ex.
2.16]{gasper}. Its equivalent forms appeared before in
\cite[p.~47]{Schwarz} and \cite[p.~236]{tannery}.
\begin{cor} \label{sigma}
For nonzero complex numbers $a$, $b$, $c$, $d$,
\begin{align}
&b(ad, q/ad, bc, q/bc, d/a, qa/d, c/b, qb/c; q)_\infty\notag \\
+&c(ab, q/ab, cd, q/cd, b/a, qa/b, d/c, qc/d; q)_\infty\notag \\
+&d(ac, q/ac, bd, q/bd, c/a, qa/c, b/d, qd/b; q)_\infty=0.
\end{align}
\end{cor}
\begin{proof} In \eqref{abcd1}, replacing $a$ with $-ad$, $b$ with $-q/ad$,  $c$ with $-bc$, and
$d$ with $-q/bc$, we have
\begin{align*}
\langle ad, bc; q\rangle_\infty=\langle -abcd,-adq/bc;
q^2\rangle_\infty-ad\langle -abcdq, -bc/ad; q^2)_\infty.
\end{align*}
We replace $a$ with $-d/a$, $b$ with $-qa/d$,  $c$ with $-c/b$, and
$d$ with $-qc/d$ in \eqref{abcd1} to deduce that
\begin{align*}
\langle d/a, c/b; q\rangle_\infty=\langle -cd/ab, -bdq/ac;
q^2\rangle_\infty-d/a\langle -abq/cd, -ac/bd; q^2\rangle_\infty.
\end{align*}
So we have \begin{align} \label{sigma1}
&b\langle ad, bc, d/a, c/b; q\rangle_\infty\notag \\
=&b\langle -abcd,-adq/bc, -cd/ab, -bdq/ac;
q^2\rangle_\infty-bd/a\langle -abcd,-adq/bc,-abq/cd, -ac/bd;
q^2\rangle_\infty\notag \\
 &-abd\langle -abcdq,
-bc/ad,-cd/ab, -bdq/ac; q^2)_\infty)+bd^2\langle -abcdq, -bc/ad,
-abq/cd, -ac/bd; q^2\rangle_\infty.
\end{align}
Similarly we have
\begin{align} \label{sigma2}
&c\langle ab, cd, b/a,d/c; q\rangle _\infty\notag \\
=&c\langle -abcd,-abq/cd, -bd/ac, -adq/bc; q^2\rangle
_\infty-bc/a\langle -abcd,-abq/cd, -bdq/ac, -ad/bc; q^2\rangle
_\infty\notag \\
&-abc\langle -abcdq,-cd/ab, -bd/ac, -adq/bc; q^2\rangle
_\infty+b^2c\langle-abcdq,-cd/ab,-bdq/ac, -ad/bc; q^2\rangle_\infty,
\end{align}
and
\begin{align} \label{sigma3}
&d\langle ac, bd,c/a, b/d; q\rangle_\infty\notag \\
=&d\langle -abcd,-acq/bd, -bc/ad, -cdq/ab;
q^2\rangle_\infty-b\langle -abcd,-acq/bd, -bcq/ad, -cd/ab;
q^2\rangle_\infty\notag \\
&-bd^2\langle -abcdq,-ac/bd, -bc/ad, -cdq/ab;
q^2\rangle_\infty+b^2d\langle -abcdq,-ac/bd, -bcq/ad, -cd/ab;
q^2\rangle_\infty.
\end{align}
Adding \eqref{sigma1}--\eqref{sigma3}, we deduce Corollary
\ref{sigma} after simplification.
\end{proof}
From Corollary \ref{k1k2cor}, by choosing \[ B=\left(
\begin{array}{rr}
1 & 2\\
-1 & 1
\end{array}
\right),
\]
we have the following corollary.
\begin{cor}  \label{cd={ab}^2}
For $ab=q$ and $cd=q^2$,
\begin{align}
f(a,b)f(c,d)=&f(ac,bd)f(a^2dq, b^2cq)+af(acq, bd/q)f(a^2dq^3, b^2c/q)\notag \\
\label{hirschhorn1} &+bf(ac/q,bdq)f(a^2d/q, b^2cq^3).
\end{align}
\end{cor}
We can rewrite Corollary \ref{cd={ab}^2} in an equivalent form. For
$a_1, a_2\neq 0$,
\begin{align}
&\langle -a_1q; q^2\rangle_\infty\langle -a_2q^2;
q^4\rangle_\infty\notag
\\
=&\langle-a_1{a_2}^{-1}q^3;q^6\rangle_\infty\langle-{a_1}^2a_2q^6;q^{12}\rangle_\infty+a_1q\langle-a_1{a_2}^{-1}q^5;q^6\rangle_\infty\langle-{a_1}^2a_2q^{10};q^{12}\rangle_\infty\notag
\\
\label{hirschhorn2}
&+{a_1}^{-1}q\langle-a_1{a_2}^{-1}q;q^6\rangle_\infty\langle-{a_1}^2a_2q^2;q^{12}\rangle_\infty.
\end{align}
This is Hirschhorn's generalization of the quintuple product
identity in \cite{hir2}.

In \eqref{hirschhorn2}, we replace $a_1$ with $aq$, $a_2$ with
$-a^{-2}$, divide both sides by $(q^4;q^4)_\infty$, and then replace
$q^2$ with $q$. We obtain the quintuple product identity in the form
of \eqref{qpi}.

Next we choose
\[
B=\left(
\begin{array}{rr}
1 & 2\\
-2 & 1
\end{array}
\right)
\]
in Theorem \ref{2 main theorem} and deduce a generalization of the
septuple product identity after simplification.
\begin{cor} [Generalized
Septuple Product Identity]\label{cor sep} For $ab=cd=q$,
\begin{align}
f(a,b)f(c,d)=&f(ac^2q, bd^2q)f(a^2dq,b^2cq)+af(ac^2q^2, bd^2)f(a^2dq^3, b^2c/q)\nonumber \\
&+bf(ac^2, bd^2q^2)f(a^2d/q, b^2cq^3)+cf(ac^2q^3, bd^2/q)f(a^2d, b^2cq^2)\nonumber \\
\label{sep} &+df(ac^2/q, bd^2q^3)f(a^2dq^2, b^2c).
\end{align}
\end{cor}
We can rewrite \eqref{septuple1} as follows. For $a_1, a_2\neq 0$,
\begin{align}
&\langle-{a_1};q^2\rangle_\infty\langle-{a_2};q^2\rangle_\infty\notag\\
=&\langle-a_1{a_2}^2q^2;
q^{10}\rangle_\infty\langle-{a_1}^{2}{a_2}^{-1}q^4;
q^{10}\rangle_\infty+a_1\langle-a_1{a_2}^2q^4;
q^{10}\rangle_\infty\langle-{a_1}^{2}{a_2}^{-1}q^8;
q^{10}\rangle_\infty \notag \\
&+a_2\langle-a_1{a_2}^2q^6;
q^{10}\rangle_\infty\langle-{a_1}^{2}{a_2}^{-1}q^2;
q^{10}\rangle_\infty+a_1a_2\langle-a_1{a_2}^2q^8;
q^{10}\rangle_\infty\langle-{a_1}^{2}{a_2}^{-1}q^6;
q^{10}\rangle_\infty
\notag \\
\label{septuple1} &+{a_1}{a_2}^{-1}q^2\langle-a_1{a_2}^2;
q^{10}\rangle_\infty\langle-{a_1}^{2}{a_2}^{-1}q^{10};
q^{10}\rangle_\infty .
\end{align}
Replacing $a_1$ with $-a$ and $a_2$ with $-a^2$ in
\eqref{septuple1}, we have the septuple product identity in the form
of \eqref{spi}.

Identity \eqref{sep} leads to two of Ramanujan's modular identities
which are very important in the study of the modular equations of
degree $5$. We cite the identities in \cite[p.~262, Entry 10 (iv)
and (v)]{Ber1}.
\begin{cor}
\begin{align}
\label{sepcor1}
\varphi^2(q)-\varphi^2(q^5)&=4qf(q, q^9)f(q^3, q^7), \\
\label{sepcor2} \psi^2(q)-q\psi^2(q^5)&=f(q, q^4)f(q^2, q^3).
\end{align}
\end{cor}
\begin{proof}
For \eqref{sepcor1}, by choosing $a=b=c=d=q$ in \eqref{sep}, we can
deduce \eqref{sepcor1} after simplification.

For \eqref{sepcor2}, by
choosing $a=c=1$, $b=d=q$ in \eqref{sep}, we have $f(1, q)-qf(1,
q^5)=4f(q, q^4)f(q^2, q^3)$, which is equivalent to \eqref{sepcor2}.
\end{proof}

On page 207 of his lost notebook, Ramanujan listed three identities
\begin{align}
P-Q&=1+\frac{f(-q^{1/5}, -\lambda
q^{2/5})}{q^{1/5}f(-{\lambda}^{10}q^{5},
-{\lambda}^{15}q^{10})},\\
\label{PQ} PQ&=1-\frac{f(-\lambda, -{\lambda}^4q^3)f(-{\lambda}^2q,
-{\lambda}^3q^2)}{f^2(-{\lambda}^{10}q^{5},
-{\lambda}^{15}q^{10})},\\ P^5-Q^5&=1+5PQ+5P^2Q^2+\frac{f(-q,
-{\lambda}^5q^2)f^5(-{\lambda}^2q,
-{\lambda}^3q^2)}{qf^6(-{\lambda}^{10}q^{5}, -{\lambda}^{15}q^{10})}
\end{align}
without specifying $P$ and $Q$. S.~H.~Son \cite{son1} first
determined the functions $P$ and $Q$. His proof of \eqref{PQ} is
equivalent to the proof of the identity
\begin{align}
S:=&f(-\lambda, -{\lambda}^4q^3)f(-{\lambda}^2q, -{\lambda}^3q^2)
\notag \\
=&f^2(-{\lambda}^{10}q^5, -{\lambda}^{15}q^{10})-\lambda
f(-{\lambda}^{5}q^4, -{\lambda}^{20}q^{11})f(-{\lambda}^{10}q^7,
-{\lambda}^{15}q^{8})\notag \\
&-{\lambda}^2qf(-{\lambda}^{5}q^4,
-{\lambda}^{20}q^{11})f(-{\lambda}^{5}q^2,
-{\lambda}^{20}q^{13})+{\lambda}^3qf(-q,
-{\lambda}^{25}q^{14})f(-{\lambda}^{10}q^7,
-{\lambda}^{15}q^{8})\notag \\
\label{son id 2} &+{\lambda}^4q^2f(-q,
-{\lambda}^{25}q^{14})f(-{\lambda}^{5}q^2, -{\lambda}^{20}q^{13}).
\end{align}
Observing that $-\lambda\cdot-{\lambda}^4q^3=-{\lambda}^2q\cdot
-{\lambda}^3q^2={\lambda}^5q^3$, and the right-hand side of
\eqref{son id 2} is the sum of five parts. We choose $a=-\lambda$,
$b=-{\lambda}^4q^3$, $c=-{\lambda}^2q$, and $d=-{\lambda}^3q^2$ in
\eqref{sep}. Then \eqref{son id 2} follows immediately.

Define the septic Rogers-Ramanujan functions
\begin{align*}
A(q): &= \sum^{\infty}_{n=0}\frac{q^{2n^2}}{(q^2; q^2)_n(-q;
q)_{2n}}=\frac{(q^7; q^7)_\infty(q^3; q^7)_\infty(q^4;
q^7)_\infty}{(q^2; q^2)_\infty}=\frac{f(-q^3, -q^4)}{f(-q^2)},\notag
\\
B(q): &= \sum^{\infty}_{n=0}\frac{q^{2n(n+1)}}{(q^2; q^2)_n(-q;
q)_{2n}}=\frac{(q^7; q^7)_\infty(q^2; q^7)_\infty(q^5;
q^7)_\infty}{(q^2; q^2)_\infty}=\frac{f(-q^2, -q^5)}{f(-q^2)},\notag
\\
C(q): &= \sum^{\infty}_{n=0}\frac{q^{2n^2}}{(q^2; q^2)_n(-q;
q)_{2n}}=\frac{(q^7; q^7)_\infty(q; q^7)_\infty(q^6;
q^7)_\infty}{(q^2; q^2)_\infty}=\frac{f(-q, -q^6)}{f(-q^2)},
\end{align*}
where the equalities are from L. J. Slater \cite{slater}. H. Hahn
\cite{hahn1} obtained many identities for the septic
Rogers-Ramanujan functions. Identities (3.1)--(3.13) of \cite{hahn1}
are special cases of Theorem \ref{main theorem}. Next we prove (3.1)
of \cite{hahn1} as an example and omit the proofs of the others.
\begin{align} \label{hahn1}
A_1B_3-qB_1C_3-C_1A_3&=0.
\end{align}
Here
\[A_n: = A(q^n), \qquad B_n: = B(q^n), \qquad C_n: = C(q^n)\]
for positive $n$.
\begin{proof}
After simplification, we can write \eqref{hahn1} as
\begin{align} \label{hahn11}
f(-q^3, -q^4)f(-q^6, -q^{15})-qf(-q^2, -q^5)f(-q^3, -q^{18})-f(-q,
-q^6)f(-q^9, -q^{12})=0.
\end{align}
Since $-q^3\cdot-q^4=-q^2. -q^5=-q\cdot-q^6=q^7$,
$-q^6\cdot-q^{15}=-q^3\cdot-q^{18}=-q^9\cdot-q^{12}=q^{21}$, by
Theorem \ref{2 main theorem}, we need to solve the system of
equations
\begin{equation}\label{eqhh}
\left\{
\begin{array}{rcrcrcrcr}
l_1b^2_{11} &+& l_2b^2_{21}   &=& 7, \\
l_1b^2_{12} &+& l_2b^2_{22}   &=& 21,\\
\end{array}
\right.
\end{equation}
with $l_i$ and $b_{ij}$ as unknown integers. We can find that
$l_1=1$, $l_2=3$. The solutions $b_{ij}$'s to equation \eqref{eqhh}
differ only by signs, and they essentially generate the same
identity. So we can choose
\begin{align*}
B=\left(
\begin{array}{rr}
2 & 3\\
-1 & 2
\end{array}
\right).
\end{align*}
We omit the details here. Note that $k=\det B=7$. As before, adding
up the contribution of each member of the ECS corresponding to $B$,
we can deduce that the left-hand side of \eqref{hahn11} is $\ds
\frac{1}{2}f(-1, -q)f(-1, -q^3)$. Since $f(-1, a)=0$ for any
$|a|<1$, \eqref{hahn11} holds.
\end{proof}
Identity \eqref{hahn1} is not a special case of Corollary
\ref{k1k2cor}. The reason is if we use Corollary \ref{k1k2cor}, then
based on $l_1=1$ and $l_2=3$, the simplest matrix
\[ B=\left(
\begin{array}{rr}
1 & 3\\
-1 & 1
\end{array}
\right)
\]
would arise, but this $B$ cannot be used to deduce \eqref{hahn1}.

Schr\"{o}ter's formula first appeared in H. Schr\"{o}ter's
dissertation in 1854. It is very useful in proving many of
Ramanujan's modular equations. Next we show that Theorem \ref{2 main
theorem} is also a generalization of Schr\"{o}ter's formula and
derive several generalized forms of Schr\"{o}ter's formula. We use
the notation of Schr\"{o}ter's formula in \cite{borwein}.

Let a general theta function be defined as: $\ds
T(x,q)=\sum^\infty_{n=-\infty}x^nq^{n^2}$, where $x\neq 0$, $|q|<1$.
It is obvious that $T(x,q)=T(1/x, q)$ from this definition.
\begin{thm}[Schr\"{o}ter's Formula]
For positive integers $a$, $b$,
\begin{align} \label{schroter1}
T(x,q^a)T(y,q^b)=\sum^{a+b-1}_{n=0}y^nq^{bn^2}T(xyq^{2bn};q^{a+b})T(x^{-b}y^{a}q^{2abn},q^{ab^2+a^2b}).
\end{align}
\end{thm}
Since $T(x, q)=f(xq, q/x)$, it is easy to see that Theorem \ref{2
main theorem} is a generalization of Schr\"{o}ter's formula. If we
use the notations of \cite{borwein}, we can rewrite Theorem \ref{2
main theorem} as below.
\begin{thm} [Generalized Schr\"{o}ter's Formula]
For positive integers $a,b$, we can find an $2\times2$ invertible
integer matrix $B$ such that $ab_{11}b_{12}+bb_{21}b_{22}=0$,
$\gcd(b_{21}, b_{11}b_{22})=1$ or $\gcd(b_{22}, b_{12}b_{21})=1$.
Let $k=\det B>0$. Then
\begin{align} \label{mgschroter1}
T(x,q^a)T(y,q^b)=\sum^{k-1}_{n=0}y^nq^{bn^2}T(x^{b_{11}}y^{b_{21}}q^{2bb_{21}n};q^{a{b_{11}}^2+b{b_{21}}^2})T(x^{b_{12}}y^{b_{22}}q^{2bb_{22}n};q^{a{b_{12}}^2+b{b_{22}}^2}).
\end{align}
\end{thm}

\begin{cor} \label{genschroter1thm} For positive integers $a,b,k_i$
$(i=1, 2)$, we have
\begin{align} \label{genschroter1}
T(x,q^a)T(y,q^b)=\sum^{k_2a+k^2_1k_2b-1}_{n=0}y^nq^{bn^2}T(xy^{k_1}q^{2k_1bn};q^{a+b{k_1}^2})T(x^{-k_1k_2b}y^{k_2a}q^{2k_2abn},q^{{k_1}^2{k_2}^2ab^2+{k_2}^2a^2b}).
\end{align}
\end{cor}

\begin{proof}
Choose $$B= \left(
\begin{array}{rr}
1& -k_1k_2b\\
k_1 & k_2a
\end{array}
\right).$$

\end{proof}
\noindent If we choose $k_1=k_2=1$ in \eqref{genschroter1}, we
obtain Schr\"{o}ter's formula, which corresponds to $$B= \left(
\begin{array}{rr}
1& -b\\
1 & a
\end{array}
\right).$$

We show that S.~Kongsiriwong and Z.-G.~Liu's Theorems 6--9 in
\cite{kong} are special cases of Corollary \ref{genschroter1thm}.
\begin{cor} [Theorems 6--9 in \cite{kong}]
Let $z$ and $q$ be complex numbers with $z\neq 0$ and $|q| < 1$.
Then, for any positive integer $k$,
\begin{align} \label{kong6}
&\sum^{[(k-1)/2]}_{j=0}(-1)^jq^{j^2}\sum^{\infty}_{m=-\infty}(-1)^{(k-1)m}q^{(k^2+k)m^2-(2j+1)km}\notag
\\
&\times\sum^{\infty}_{n=-\infty}q^{(k+1)n^2}z^{(k+1)n}(q^{(2j-k)n}z^j-q^{(j-2k)n}z^{k-j})\notag
\\
=&(q^2;q^2)_\infty(zq;q^2)_\infty({z}^{-1}q;q^2)_\infty(q^{2k};q^{2k})_\infty({z}^k;q^{2k})_\infty({z}^{-k}q^{2k};q^{2k})_\infty,
\\
\label{kong7}
&\frac{1-(-1)^k}{2}(-1)^{(k+1)/2}q^{(k+1)^2/4}\sum^{\infty}_{m=-\infty}(-1)^{(k-1)m}q^{(k^2+k)m^2+(k^2+k)m}\notag
\\
&\times
\sum^{\infty}_{n=-\infty}q^{(k+1)n^2-(k+1)n}z^{(k+1)n-(k+1)/2}\notag
\\
&+\sum^{\infty}_{m=-\infty}(-1)^{(k-1)m}q^{(k^2+k)m^2}\sum^{\infty}_{n=-\infty}q^{(k+1)n^2}z^{(k+1)n}\notag
\\
&+\sum^{[{k}/2]}_{j=1}(-1)^jq^{j^2}\sum^{\infty}_{m=-\infty}(-1)^{(k-1)m}q^{(k^2+k)m^2-2jkm}
\sum^{\infty}_{n=-\infty}q^{(k+1)n^2}z^{(k+1)n}(q^{2jn}z^j-q^{-2jn}z^{-j})\notag
\\
=&(q^2;q^2)_\infty(zq;q^2)_\infty({z}^{-1}q;q^2)_\infty(q^{2k};q^{2k})_\infty({z}^kq^k;q^{2k})_\infty({z}^{-k}q^{k};q^{2k})_\infty,
\\ \label{kong8}
&\sum^{k}_{j=-k+1}(-1)^jq^{j^2-j}\sum^{\infty}_{m=-\infty}(-1)^{m}q^{(4k^2+k)m^2+(1-4j)km}\notag
\\
&\times \sum^{\infty}_{n=-\infty}(-1)^n
q^{(4k+1)n^2}z^{(4k+1)n}(q^{(2j-2k-1)n}z^j-q^{(2k+1-2j)n}z^{2k+1-j})\notag
\\
=&(q^2;q^2)_\infty(z;q^2)_\infty({z}^{-1}q^2;q^2)_\infty(q^{2k};q^{2k})_\infty({z}^{2k};q^{2k})_\infty({z}^{-2k}q^{2k};q^{2k})_\infty,
\\
 \label{kong9}
&\sum^{\infty}_{m=-\infty}(-1)^{m}q^{(4k^2+k)m^2}\sum^{\infty}_{n=-\infty}(-1)^nq^{(4k+1)n^2}z^{(4k+1)n}\notag
\\
&+\sum^{2k}_{j=1}(-1)^jq^{j^2}\sum^{\infty}_{m=-\infty}(-1)^{m}q^{(4k^2+k)m^2-4jkm}\notag
\sum^{\infty}_{n=-\infty}(-1)^nq^{(4k+1)n^2}z^{(4k+1)n}(q^{2jn}z^j+q^{-2jn}z^{-j})\notag
\\
=&(q^2;q^2)_\infty(zq;q^2)_\infty({z}^{-1}q;q^2)_\infty(q^{2k};q^{2k})_\infty({z}^{2k}q^k;q^{2k})_\infty({z}^{-2k}q^{k};q^{2k})_\infty.
\end{align}
\begin{proof}
For \eqref{kong6} and \eqref{kong7}, choose the matrix \[ B=\left(
\begin{array}{rr}
1 & -k\\
1 & 1
\end{array}
\right),
\]
this corresponds to the case $a=1$ and $b=k$ in Schr\"{o}ter's
formula \eqref{schroter1}.

For \eqref{kong8} and \eqref{kong9}, choose \[ B=\left(
\begin{array}{rr}
1 & -2k\\
2 & 1
\end{array}
\right),
\]
this corresponds to $a=1$, $b=k$, $k_1=2$, and $k_2=1$ in
\eqref{genschroter1}. We omit the details.
\end{proof}
\end{cor}
If we choose $k_1=1$ and $k_2=k$ in \eqref{genschroter1}, we deduce
the following corollary.
\begin{cor}
For positive integers $a,b,k$,
\begin{align} \label{mult k}
T(x,q^a)T(y,q^b)=\sum^{ka+kb-1}_{n=0}y^nq^{bn^2}T(xyq^{2bn};q^{a+b})T(x^{-kb}y^{ka}q^{2kabn},q^{k^2ab(a+b)}).\end{align}
\end{cor}
\noindent We can obtain \eqref{mult k} directly from Schr\"{o}ter's
formula \eqref{schroter1} by replacing $q$ with $q^{1/k}$, $a$ with
$ka$, and $b$ with $kb$.

If we choose $k_1=k$ and $k_2=1$ in \eqref{genschroter1}, then we
arrive at the next corollary.
\begin{cor}
For positive integers $a,b,k$,
\begin{align} \label{mult k2}
T(x,q^a)T(y,q^b)=\sum^{a+k^{2}b-1}_{n=0}y^nq^{bn^2}T(xy^kq^{2bkn},q^{a+bk^2})T(x^{-kb}y^{a}q^{2abn},q^{k^2ab^2+a^2b}).
\end{align}
\end{cor}
If we choose $k=2$, $a=3$, $b=1$, $x=-q^{-3}$, $y=-q^{-1}$ in
\eqref{mult k2}, and then replace $q^2$ with $q$, we can obtain
Hahn's identity \eqref{hahn1}.

Assuming that $k_2|a$, we can also derive a variation of
\eqref{genschroter1} by replacing $k_2$ with $1/k_2$.
\begin{cor}
For positive integers $a,b,k_i$ $(i=1, 2)$, $k_2|k_1b$, and $k_2|a$,
we have
\begin{align} \label{general schroter chu}
&T(x,q^a)T(y,q^b)\notag \\
=&\sum^{a/k_2+k^2_1b/k_2-1}_{n=0}y^nq^{bn^2}T(xy^{k_1}q^{2k_1bn};q^{a+b{k_1}^2})T(x^{-k_1b/k_2}y^{a/k_2}q^{2abn/k_2},q^{ab^2{k_1}^2/{k_2}^2+a^2b/{k_2}^2}).
\end{align}
\end{cor}
We obtain the next corollary if we choose the transformation matrix
$$
B = \left(
\begin{array}{rr}
1& -k_1k_2b\\
k_1a & k_2
\end{array}
\right).$$
\begin{cor} \label{schrotercy}
For positive integers $a,b,k_i(i=1, 2)$, we have
\begin{align} \label{schrotercy1}
&T(x,q^a)T(y,q^b)\notag \\
=&\sum^{k_2+k^2_1k_2ab-1}_{n=0}y^nq^{bn^2}T(xy^{k_1a}q^{2k_1abn};q^{a+b{k_1}^2a^2})T(x^{-k_1k_2b}y^{k_2}q^{2k_2abn},q^{{k_1}^2{k_2}^2ab^2+{k_2}^2b}).
\end{align}
\end{cor}
We can obtain Corollary \ref{schrotercy} by replacing $k_1$ with
$k_1a$ and $k_2a$ with $k_2$ in Corollary \ref{genschroter1thm}.

In \cite{chu}, W. Chu and Q. Yan  obtained a general formula with
many known identities as special cases. The main Theorem in Chu and
Yan's paper is the following result.
\begin{cor}
Let $\alpha, \beta, \gamma$ be three natural integers with
$gcd(\alpha, \gamma)=1$ and $\lambda=1+\alpha\beta^2\gamma.$ For two
indeterminates $x$ and $y$ with $x\neq 0$ and $y\neq 0$, there holds
the algebraic identity
\begin{align}
\langle x;q^\alpha\rangle_\infty \langle x^{\beta\gamma}y;
q^\gamma\rangle_\infty=&\sum^{\alpha\beta^2\gamma}_{l=0}(-1)^lq^{{l
\choose 2}\alpha} x^l\langle (-1)^{\alpha\beta} x^\lambda
y^{\alpha\beta} q^{{\alpha\beta \choose 2}\gamma+l\alpha};
q^{\lambda\alpha}\rangle\notag \\
\label{schroter chu} &\times\langle (-1)^{\beta\gamma} y
q^{{\beta\gamma+1 \choose 2}\gamma-l\alpha\beta\gamma};
q^{\lambda\gamma}\rangle.
\end{align}
\end{cor}
This corollary is a special case of both \eqref{general schroter
chu} and \eqref{schrotercy1}. If we choose $k_1=\alpha\beta$,
$k_2=\alpha$, $a=\alpha$,  and $b=\gamma$, replace $x$ with
$-q^{-\alpha/2} x$, $y$ with $-x^{\beta \gamma}yq^{-\gamma/2}$, $q$
with $q^{1/2}$ in \eqref{general schroter chu}, and finally change
the notation $n$ to $l$, we obtain \eqref{schroter chu}. Or if we
choose $k_1=\beta$, $k_2=1$, $a=\alpha$, and $b=\gamma$, replace $x$
with $-q^{-\alpha/2} x$, $y$ with $-x^{\beta \gamma}yq^{-\gamma/2}$,
$q$ with $q^{1/2}$ in \eqref{schrotercy1}, and finally change the
notation $n$ to $l$ to obtain \eqref{schroter chu}. Identity
\eqref{schroter chu} is corresponding to $$B= \left(
\begin{array}{rr}
\beta\gamma& -1\\
1 & \alpha\beta
\end{array}
\right).$$

If we choose $k_1=k_2=k$ in \eqref{general schroter chu}, we can
deduce the next corollary.
\begin{cor}
For positive integers $a,b,k$, such that $k|a$,
\begin{align} \label{general schroter chu1}
T(x,q^a)T(y,q^b)=\sum^{\frac{a}{k}+bk-1}_{n=0}y^nq^{bn^2}T(xy^kq^{2bkn},q^{a+bk^2})T(x^{-b}y^{\frac{a}{k}}q^{\frac{2ab}{k}n},q^{ab^2+\frac{a^2b}{k^2}}).
\end{align}
\end{cor}
The Blecksmith-Brillhart-Gerst theorem in \cite[Theorem 2]{BBG1},
which is a nice generalization of Schr\"{o}ter's formula,  can be
shown to be a special case of \eqref{general schroter chu1}. We cite
the theorem in \cite[p.~73]{Ber1}, which uses Ramanujan's notation.
\begin{cor}
Define $f_0(a, b)=f(a, b)$ and $f_1(a, b)= f(-a, -b)$. Let $a$, $b$,
$c$ and $d$  denote complex numbers such that $|ab|$, $|cd| < 1$.
Suppose that there exist positive integers $\alpha$, $\beta$, and
$m$ such that
\[(ab)^\beta=(cd)^{\alpha(m-\alpha\beta)}.\]
Let $\varepsilon_1, \varepsilon_2 \in \{0, 1\}$. Then
\begin{align*}
&f_{\varepsilon_1}(a, b)f_{\varepsilon_2}(c, d)\\
=&\sum_{r\in
R}(-1)^{\varepsilon_2r}{c}^{r(r+1)/2}{d}^{r(r-1)/2}f_{\delta_1}\left(\frac{a(cd)^{\alpha(\alpha+1-2n)/2}}{c^\alpha},
\frac{b(cd)^{\alpha(\alpha+1+2n)/2}}{d^\alpha} \right) \\
&\times
f_{\delta_2}\left(\frac{(a/b)^{\beta/2}(cd)^{(m-\alpha\beta)(m+1+2n)/2}}{d^{m-\alpha\beta}},\frac{(b/a)^{\beta/2}(cd)^{(m-\alpha\beta)(m+1-2n)/2}}{c^{m-\alpha\beta}}\right),
\end{align*}
where $R$ is a complete residue system $(\hbox{mod}\ m$),
$$\delta_1=
\left\{
\begin{array}{rcrcrcrcr}
0,   \qquad \ &\text{if} \ \varepsilon_1 + \alpha \varepsilon_2 \ \text{is} \ even,\\
1,   \qquad \ &\text{if} \ \varepsilon_1 + \alpha \varepsilon_2 \ \text{is} \ odd, \\
\end{array}
\right.
$$
and
$$\delta_2= \left\{
\begin{array}{rcrcrcrcr}
0,   \qquad \ &\text{if} \ \varepsilon_1 \beta + \varepsilon_2(m - \alpha \beta) \ \text{is} \ even, \\
1,   \qquad \ &\text{if} \ \varepsilon_1 \beta + \varepsilon_2(m - \alpha \beta) \ \text{is} \ odd. \\
\end{array}
\right.
$$
\end{cor}

\begin{proof}
Replacing $x$ with $(-1)^{\varepsilon_1}x$, and $y$ with
$(-1)^{\varepsilon_2}y$ in \eqref{general schroter chu1}, we have
\begin{align}
T\left((-1)^{\varepsilon_1}x,q^a\right)T\left((-1)^{\varepsilon_2}y,q^b\right)=&\sum^{\frac{a}{k}+bk-1}_{n=0}(-1)^{\varepsilon_2
n}y^nq^{bn^2}T\left((-1)^{\varepsilon_1+k\varepsilon_2}xy^kq^{2bkn},q^{a+bk^2}\right)\notag
\\
\label{general schroter chu2} &\times
T\left((-1)^{b\varepsilon_1+\frac{a}{k}\varepsilon_2}
x^{-b}y^{\frac{a}{k}}q^{\frac{2ab}{k}n},q^{ab^2+\frac{a^2b}{k^2}}\right).
\end{align}
We replace $x$ with $(b/a)^{1/2}$, $q^a$ with $(ab)^{1/2}$
 , $y$ with $(c/d)^{1/2}$, $q^b$ with
$(cd)^{1/2}$, and $k$ with $\alpha$ in \eqref{general schroter
chu2}, where $a, b, c, d$ are complex numbers such that $|ab|<1,
|cd|<1$ and $\alpha, \beta, m$ are positive integers. We choose
$a={\alpha(m-\alpha\beta)}$ and $b={\beta}$ such that
$(ab)^\beta=(cd)^{\alpha(m-\alpha\beta)}$. By \eqref{general
schroter chu2},
\begin{align*}
&T\left((-1)^{\varepsilon_1}(b/a)^{1/2},(ab)^{1/2}\right)T\left((-1)^{\varepsilon_2}(c/d)^{1/2},(cd)^{1/2}\right)\\
=&\sum^{m-1}_{n=0}(-1)^{\varepsilon_2}(c/d)^{n/2}(cd)^{n^2/2}T\left((-1)^{\varepsilon_1+\alpha
\varepsilon_2}(b/a)^{1/2}(c/d)^{\alpha/2}(cd)^{\alpha
n};(ab)^{1/2}(cd)^{\alpha^2/2}\right)\\
&\times
T\left((-1)^{\beta\varepsilon_1+(m-\alpha\beta)\varepsilon_2}(a/b)^{\beta/2}(c/d)^{(m-\alpha\beta)/2}(cd)^{m-\alpha\beta
n},(ab)^{\beta^2/2}(cd)^{(m-\alpha\beta)^2/2}\right).
\end{align*}
After simplification, we have
\begin{align*}
&f\big((-1)^{\varepsilon_1}a,(-1)^{\varepsilon_1}b\big)f\big((-1)^{\varepsilon_2}c,(-1)^{\varepsilon_2}d\big)\\
=&\sum^{m-1}_{n=0}(-1)^{\varepsilon_2}{c}^{n(n+1)/2}{d}^{n(n-1)/2}f\left((-1)^{\varepsilon_1+\alpha
\varepsilon_2}\frac{b(cd)^{\alpha(\alpha+1+2n)/2}}{d^\alpha},
(-1)^{\varepsilon_1+\alpha
\varepsilon_2}\frac{a(cd)^{\alpha(\alpha+1-2n)/2}}{c^\alpha}\right)\\
&\times
f\Big((-1)^{\beta\varepsilon_1+(m-\alpha\beta)\varepsilon_2}\frac{(a/b)^{\beta/2}(cd)^{(m-\alpha\beta)(m+1+2n)/2}}{d^{m-\alpha\beta}},\\
&(-1)^{\beta\varepsilon_1+(m-\alpha\beta)\varepsilon_2}\frac{(b/a)^{\beta/2}(cd)^{(m-\alpha\beta)(m+1-2n)/2}}{c^{m-\alpha\beta}}\Big).
\end{align*}
Here the corresponding matrix is
$$
B = \left(
\begin{array}{rr}
1& -\beta\\
\alpha & m-\alpha\beta
\end{array}
\right).$$

\end{proof}

If we choose $k_1=1, k_2=k$ in \eqref{general schroter chu}, we can
deduce the following corollary.
\begin{cor}\label{schroter div cor}
For positive integers $a,b,k$, such that $k|a$ and $k|b$,
\begin{align} \label{schroter div1}
T(x,q^a)T(y,q^b)=\sum^{\frac{a+b}{k}-1}_{n=0}y^nq^{bn^2}T(xyq^{2bn},q^{a+b})T(x^{-\frac{b}{k}}y^{\frac{a}{k}}q^{\frac{2ab}{k}n},q^{\frac{ab(a+b)}{k^2}}).
\end{align}
\end{cor}
We can obtain Corollary \ref{schroter div cor} by replacing $q$ with
$q^{k}$, $a$ with $a/k$, and $b$ with $b/k$ in Schr\"{o}ter's
formula \eqref{schroter1}.

For each of the above results related to Schr\"{o}ter's formula, we
can switch the positions of $x$ and $y$, and $a$ and $b$ at the same
time since their positions are symmetric. We can also replace $x$
with $1/x$ and $y$ with $1/y$ on the right-hand sides of all such
formulas.

The general approach we developed previously can also be applied to
some infinite sums that are not theta functions. The idea is for any
infinite sum, we want all the coefficients of $y_iy_j$ ($i\neq j$)
terms to be zero after changing the variables from $x_1, \dots, x_n$
to $y_1, \dots, y_n$. For example, we can consider the following
infinite sum
\begin{align} \label{specase}
\sum^{\infty}_{x_1, x_2=-\infty}q^{{x_1}^2+x_1x_2+{x_2}^2}.
\end{align}
We change the variables from $x_1, x_2$ to $y_1, y_2$ by the
integral matrix exact covering system
$$\bigg\{ \left(
\begin{array}{rr}
x_1\\
x_2
\end{array}
\right)= \left(
\begin{array}{rr}
1& 1\\
-1 & 1
\end{array}
\right)\left(
\begin{array}{rr}
y_1\\
y_2
\end{array}
\right),\quad \left(
\begin{array}{rr}
x_1\\
x_2
\end{array}
\right)= \left(
\begin{array}{rr}
1& 1\\
-1 & 1
\end{array}
\right)\left(
\begin{array}{rr}
y_1\\
y_2
\end{array}
\right)+\left(
\begin{array}{rr}
1\\
0
\end{array}
\right)\bigg\}.$$ So we can write the infinite sum \eqref{specase}
as the linear combination of two parts
\begin{align*}
&\sum^{\infty}_{x_1, x_2=-\infty}q^{{x_1}^2+{x_1}x_2+{x_2}^2}\notag \\
=&\sum^{\infty}_{y_1,
y_2=-\infty}q^{(y_1+y_2)^2+(y_1+y_2)(-y_1+y_2)+(-y_1+y_2)^2}+\sum^{\infty}_{y_1,
y_2=-\infty}q^{(y_1+y_2+1)^2+(y_1+y_2+1)(-y_1+y_2)+(-y_1+y_2)^2}\notag
\\
=&\sum^{\infty}_{y_1,
y_2=-\infty}q^{{y_1}^2+3{y_2}^2}+\sum^{\infty}_{y_1,
y_2=-\infty}q^{{y_1}^2+y_1+3{y_2}^2+3y_2}\notag
\\
=&\varphi(q)\varphi(q^3)+4q\psi(q^2)\psi(q^6).
\end{align*}
For the general case, we consider $\ds \sum^{\infty}_{x_1,
x_2=-\infty}q^{a{x_1}^2+bx_1x_2+c{x_2}^2}$, where $a, b, c$ are
positive integers. We choose the matrix $B$ in the ECS such that the
coefficients of $y_1y_2$ in $a(b_{11}y_1+
b_{12}y_2)^2+b(b_{11}y_1+b_{12}y_2)(b_{21}y_1+
b_{22}y_2)+c(b_{21}y_1+ b_{22}y_2)^2$ to be zero. So the entries of
the matrix $B$ must satisfy the requirement
\begin{align*}
2a b_{11}b_{12}+b({b_{11}b_{22}+b_{12}b_{21})+2c b_{21}b_{22}=0}.
\end{align*}
Accordingly, we prove an identity in \cite{shen}  by L.-C. Shen.
\begin{cor}
\begin{align*}
(q^2;
q^2)^2_\infty=\sum^{\infty}_{m=-\infty}\sum^{\infty}_{n=-\infty}(-1)^m
q^{2m^2+2mn+2n^2+n}.
\end{align*}
\end{cor}
\begin{proof}
This identity is equivalent to
\begin{align*} (q^2;
q^2)^2_\infty=\psi(-q)\{f(-q^5, -q^7)-qf(-q,-q^{11})\}
\end{align*}
if we choose matrix
\[ B=\left(
\begin{array}{rr}
1 & -1\\
1 & 1
\end{array}
\right).
\]
By \eqref{entry 311}, $f(-q, q^2)=f(-q^5, -q^7)-qf(-q,-q^{11})$. So
we only need to show that $(q^2; q^2)^2_\infty=\psi(-q)f(-q, q^2)$.
Recall an identity in \cite[p.~11]{spirit}
\begin{align} \label{pshen}
\psi(q)=\frac{(q^2, q^2)_\infty}{(q; q^2)_{\infty}}.
\end{align}
Noticing that \[\frac{(q^2; q^2)_\infty}{(q;
q^2)_{\infty}}=\frac{(q^2; q^2)^2_\infty}{(q; q)_{\infty}},\] we
replace $q$ with $-q$ in \eqref{pshen} and finish the proof.
\end{proof}

\section{Products of Three or More Theta Functions}
We consider $$B=\left[
\begin{array}{rrr}
1 & 1 & 1  \\
2 & -1 & 0 \\
1 & 1 & -1
\end{array}
\right].$$ The set of the 3 columns of $B$ is an orthogonal set and
$\det B=6$.

Since $$B^*=\left[
\begin{array}{rrr}
1 & 2 & 1  \\
2 & -2 & 2 \\
3 & 0 & -3
\end{array}
\right], $$ by Theorem \ref{main theorem}, the integer matrix ECS
$\Bigg\{B\Z^3+\left(
\begin{array}{rrrr}
r\\
0\\
0
\end{array}
\right)\Bigg\}_{r=0}^{5}$ covers $\Z^3$.

Consider
\begin{align*}
S&:=\prod_{i=1}^3f(a_i, b_i)=\sum^\infty_{x_1, x_2,
x_3=-\infty}\prod_{i=1}^3{a_i}^{\frac{{x_i}^2+x_i}{2}}{b_i}^{\frac{{x_i}^2-x_i}{2}}.
\end{align*}
Corresponding to the above ECS, we conclude the following special
case of Theorem \ref{main theorem}.
\begin{cor}
For $q=a_1b_1=a_2b_2=a_3b_3$,
\begin{align}
\prod _{i=1}^3 f(a_i,b_i)=&f(a_1{a_2}^2a_3q,
b_1{b_2}^2b_3q)f(a_1b_2a_3,b_1a_2b_3)f(a_1b_3, b_1a_3)\notag
\\
&+a_1f(a_1{a_2}^2a_3q^2, b_1{b_2}^2b_3)f(a_1b_2a_3q,
b_1a_2b_3/q)f(a_1b_3q, b_1a_3/q)\notag \\
&+a_2f(a_1{a_2}^2a_3q^3, {b_1{b_2}^2b_3}/q)f(a_1b_2a_3/q,
b_1a_2b_3q)f(a_1b_3, b_1a_3)\notag \\
&+a_1a_2f(a_1{a_2}^2a_3q^4, {b_1{b_2}^2b_3}/q^2)f(a_1b_2a_3,
b_1a_2b_3)f(a_1b_3q, b_1a_3/q)\notag \\
&+a_1a_2a_3f(a_1{a_2}^2a_3q^5, {b_1{b_2}^2b_3}/q^3)f(a_1b_2a_3q,
b_1a_2b_3/q)f(a_1b_3, b_1a_3)\notag \\
\label{abcdef} &+b_3f(a_1{a_2}^2a_3,
{b_1{b_2}^2b_3}q^2)f(a_1b_2a_3/q, b_1a_2b_3q)f(a_1b_3q, b_1a_3/q).
\end{align}
\end{cor}
If we begin with the solution set of the system of congruences, we
need consider $B'=6B$. It is easy to verify that $\det
B'=1296={36}^{3-1}$ and $36\cdot {B'}^{-1}$ is an integer matrix. So
$k'=36$. Now $B'y \equiv 0$ $(\hbox{mod}\ 36)$ is equivalent to $By
\equiv 0$ $(\hbox{mod}\ 6)$. By changing variables from $x_i$ to
$y_i$ in $S$ by the transformation $y=Ax$, where
$$A=\left[
\begin{array}{rrr}
1 & 2 & 1  \\
2 & -2 & 2 \\
3 & 0 & -3
\end{array}
\right],
$$
we have $By \equiv 0$ $(\hbox{mod}\ 6)$. So for the system of
congruences
\begin{align*}
\left\{
\begin{array}{rcrcrcrcr}
y_1 &+& y_2 &+& y_3 &\equiv 0\enspace(\hbox{mod}\ 6), \\
2y_1&-& y_2  &\ & &\equiv 0\enspace(\hbox{mod}\ 6), \\
y_1 &+& y_2 &-& y_3 &\equiv 0\enspace(\hbox{mod}\ 6), \\
\end{array}
\right.
\end{align*}
the solution set has six solutions modulo 6. We can also derive
\eqref{product3theta} by adding up the contribution of each solution
to the sum $S$.

We can rewrite \eqref{product3theta} as follows. For $a_1, a_2, a_3
\neq 0$,
\begin{align}
&\prod_{i=1}^3 \langle-a_i;q\rangle_\infty\notag
\\
=&\langle-a_1{a_2}^2a_3q;
q^6\rangle_\infty\langle-a_1{a_2}^{-1}a_3q;
q^3\rangle_\infty\langle-{a_1}{a_3}^{-1}q;
q^2\rangle_\infty\notag \\
&+a_1\langle-a_1{a_2}^2a_3q^2;
q^6\rangle_\infty\langle-a_1{a_2}^{-1}a_3q^2;
q^3\rangle_\infty\langle-{a_1}{a_3}^{-1}q^2; q^2\rangle_\infty \notag \\
&+a_2\langle-a_1{a_2}^2a_3q^3;
q^6\rangle_\infty\langle-a_1{a_2}^{-1}a_3;
q^3\rangle_\infty\langle-{a_1}{a_3}^{-1}q; q^2\rangle_\infty\notag \\
&+a_1a_2\langle-a_1{a_2}^2a_3q^4;
q^6\rangle_\infty\langle-a_1{a_2}^{-1}a_3q;
q^3\rangle_\infty\langle-{a_1}{a_3}^{-1}q^2; q^2\rangle_\infty\notag \\
&+a_1a_2a_3\langle-a_1{a_2}^2a_3q^5;
q^6\rangle_\infty\langle-a_1{a_2}^{-1}a_3q^2;
q^3\rangle_\infty\langle-{a_1}{a_3}^{-1}q; q^2\rangle_\infty\notag
\\
\label{product3theta}
 &+{a_3}^{-1}q\langle-a_1{a_2}^2a_3;
q^6\rangle_\infty
 \langle-a_1{a_2}^{-1}a_3;
q^3\rangle_\infty\langle-{a_1}{a_3}^{-1}q^2; q^2\rangle_\infty.
\end{align}

\noindent Replacing $q$ with $q^2$, and then letting
$a_1=b_1=a_2=b_2=a_3=b_3=q$ in \eqref{abcdef}, we deduce the
following corollary.
\begin{cor}
\begin{align*}
{\varphi^3}(q)=&\varphi(q^2)\varphi(q^3)\varphi(q^6)+4q\psi(q^4)f(q,q^5)f(q^4,
q^8)\notag \\
&+2q\varphi(q^2)f(q, q^5)f(q^2,
q^{10})+4q^2\varphi(q^3)\psi(q^4)\psi(q^{12}).
\end{align*}
\end{cor}
Letting $a_1=a_2=a_3=1, b_1=b_2=b_3=q$ in \eqref{abcdef}, we deduce
the next corollary.
\begin{cor}
\begin{align*}
4{\psi^3}(q)=&\phi(q)f(q,q^2)f(q, q^5)+2\psi(q^2)f(q, q^2)f(q^2,
q^4)\notag \\
&+\varphi(q)\varphi(q^3)\psi(q^3)+4q\psi(q^2)\psi(q^3)\psi(q^6).
\end{align*}
\end{cor}
Letting $a_3=-a_1$ in \eqref{product3theta}, we deduce Corollary
\ref{cd={ab}^2} after simplification.

Replacing each $a_i$ with $-a_i$, then replacing $a_1$ with $a$,
$a_2$ with $ab$, and $a_3$ with $b$ in \eqref{product3theta} and
simplifying, we deduce the following analogue of Winquist's
identity.
\begin{cor} \label{ana winquist}
For $a, b$ nonzero,
\begin{align}
&(a,{a}^{-1}q,b,{b}^{-1}q,ab,{a}^{-1}{b}^{-1}q,q,q;
q)^2_\infty\notag \\
=&\langle-a{b}^{-1}q; q^2\rangle_\infty\big[\langle-{a}^3{b}^3q;
q^6\rangle_\infty-a^2b^2\langle-{a}^3{b}^3q^5;
q^6\rangle_\infty\big]\notag
\\ \label{swinquist}
&+\langle-a{b}^{-1}q^2;
q^2\rangle_\infty\big[a^2b\langle-{a}^3{b}^3q^4;
q^6\rangle_\infty-a\langle-{a}^3{b}^3q^2; q^6\rangle_\infty\big].
\end{align}
\end{cor}
\noindent The product in Corollary \ref{ana winquist} is the product
that occurs in MacDonald's identities \cite{macdonald} for $A_2$.

Recently, various representations of $(q; q)^{10}_\infty$ have been
given in \cite{BCLY},  \cite{liu1},\cite{chu1},
\cite{chanliung},\cite{shchan}, and \cite{chuyan}. They are used in
the proofs of Ramanujan's congruence $p(11n+6)\equiv 0 \pmod{11}$.
In \cite{cooper1}, S.~Cooper, H-H.~Chan, and P.C.~Toh gave two
representations of $(q; q)^8_\infty$. Here we derive a new
expression of $(q; q)^8_\infty$.
\begin{cor} \label{q8}
\begin{align}
16(q;q)^{8}_\infty=\sum^{\infty}_{m,n=-\infty}\big[(2n-6m+1)(2n+6m-2)^2q^{-2m}-(2n-6m-3)(2n+6m+2)^2q^{2m}
\notag \\
+(2n-6m-1)(2n+6m+2)^2q^{n+m}-(2n-6m+1)(2n+6m)^2q^{n-m}\big]q^{n^2+3m^2}.
\end{align}
\end{cor}
\begin{proof}
Replacing $a$ with $a^2$, $b$ with $b^2$, and multiplying both sides
of Corollary \ref{ana winquist} by $a^{-1}b^{-1}$, we obtain
\begin{align} \label{qq81}
&(a-1/a)(b-1/b)(a^2q;q)_\infty({a}^{-2}q;q)_\infty(b^2q;q)_\infty({b}^{-2}q;q)_\infty(a^2b^2;q)_\infty({a}^{-2}{b}^{-2}q;q)_\infty(q;
q)^2_\infty\notag \\
=&\sum^{\infty}_{m,n=-\infty}\big[a^{2n+6m-1}b^{-2n+6m-1}q^{-2m}-a^{2n+6m+3}b^{-2n+6m+3}q^{2m}\notag
\\
&+a^{2n+6m+3}b^{-2n+6m+1}q^{n+m}-a^{2n+6m+1}b^{-2n+6m-1}q^{n-m}\big]q^{n^2+3m^2}.
\end{align}
We differentiate \eqref{qq81} with respect to $b$, let $b=1$, and
then multiply both sides by $a^{-1}$ to obtain
\begin{align} \label{q81}
&-2(a-1/a)^2(a^2q;q)^2_\infty({a}^{-2}q;q)^2_\infty(q;q)^4_\infty\notag
\\
=&\sum^{\infty}_{m,n=-\infty}\big[(-2n+6m-1)a^{2n+6m-2}q^{-2m}-(-2n+6m+3)a^{2n+6m+2}q^{2m}\notag
\\
&+(-2n+6m+1)a^{2n+6m+2}q^{n+m}-(-2n+6m-1)a^{2n+6m}q^{n-m}\big]q^{n^2+3m^2}.
\end{align}
Applying the differential operator $\ds a\frac{d}{da}$ twice, and
letting $a=1$ in \eqref{q81}, we deduce Corollary \ref{q8}.
\end{proof}
Replacing $q$ by $q^5$ in \eqref{swinquist}, and then letting $a=q$,
$b=q^3$, replacing $q$ by $q^5$ in \eqref{swinquist}, and then
choosing $a=q$, $b=q^2$, we deduce the following results.
\begin{cor}
\begin{align*}
f(-q)f(-q, -q^4)&=f(q^3, q^7)\big[f(q^{13}, q^{17})-q f(q^{7},
q^{23})\}+f(q^2, q^8)\{q^3f(q^{2}, q^{28})-q f(q^{8}, q^{22})\big].\\
f(-q)f(-q^2, -q^3)&=f(q^4, q^6)\big[f(q^{14}, q^{16})-q^2f(q^{4},
q^{26})\}+f(q, q^9)\{q^4f(q, q^{29})-qf(q^{11}, q^{19})\big].
\end{align*}
\end{cor}
Letting $b=-1$ in \eqref{swinquist}, we deduce the next corollary.
\begin{cor}
For $a, b$ nonzero,
\begin{align}
&2\langle a^2;q^2\rangle_\infty\notag \\
=&(aq, {a}^{-1}q; q^2)_\infty\big[\langle{a}^3q;
q^6\rangle_\infty-a^2\langle{a}^3q^5;
q^6\rangle_\infty\big]-(aq^2,{a}^{-1};
q^2)_\infty\big[a\langle{a}^3q^2;
q^6\rangle_\infty+a^2\langle{a}^3q^4; q^6\rangle_\infty\big].
\end{align}
\end{cor}
Letting $b=-a$ in \eqref{swinquist}, we deduce the next corollary.
\begin{cor} For nonzero complex number $a$,
\begin{align}
\langle a^2;q^2\rangle_\infty(-a^2,
-{a}^{-2}q;q)_\infty=\langle{a}^6q;
q^6\rangle_\infty-a^4\langle{a}^6q^5; q^6\rangle_\infty.
\end{align}
\end{cor}
Letting $b=1$ in \eqref{swinquist}, we deduce the next corollary.
\begin{cor}
For nonzero complex number $a$,
\begin{align}
&\langle-aq; q^2\rangle_\infty\big[\langle-{a}^3q;
q^6\rangle_\infty-a^2\langle-{a}^3q^5;
q^6\rangle_\infty\big]\notag \\
+&\langle-aq^2; q^2\rangle_\infty\big[\langle a^2(-{a}^3q^4;
q^6\rangle_\infty-a\langle-{a}^3q^2; q^6\rangle_\infty\big]=0.
\end{align}
\end{cor}
The next simplest orthogonal matrix with $l_1=l_2=l_3$ is $$B=\left[
\begin{array}{rrr}
1 & 1 & 3  \\
1 & -1 & -3 \\
0 & 3 & -2
\end{array}
\right]
$$
with $\det B=22$. We skip the identity generated by this matrix
here.

 If we do not require that all the $l_i$ are the same, we
can obtain other ``generalized orthogonal'' matrices. For example,
if we choose $l_1=l_2=1$, $l_3=2$, then for the following
``generalized orthogonal'' matrix
$$B=\left[
\begin{array}{rrr}
1 & 1 & 1  \\
1 & -1 & 1 \\
1 & 0 & -1
\end{array}
\right],
$$
since $\det B=4$, we can write a product of three theta functions as
the linear combination of four products of three theta functions.
Omitting the details, we deduce the following corollary.
\begin{cor} \label{cor3case2}
For $a_1b_1=a_2b_2=q$ and $a_3b_3=q^2$,
\begin{align}
&\prod _{i=1}^3 f(a_i,b_i)
\notag \\
=&f(a_1b_2,b_1a_2)\notag \\
&\times\big[f(a_1{a_2}a_3, b_1{b_2}b_3)f(a_1a_2b_3,
b_1b_2a_3)+a_1a_2f(a_1{a_2}a_3q^2,b_1{b_2}b_3/q^2)f(a_1a_2b_3q^2,
b_1b_2a_3/q^2)\big]\notag \\
&+{a_1}f(a_1b_2q,b_1a_2/q)\\
&\times\big[f(a_1{a_2}a_3q, b_1{b_2}b_3/q)f(a_1a_2b_3q,
b_1b_2a_3/q)+a_3f(a_1{a_2}a_3q^3,
b_1{b_2}b_3/q^3)f(a_1a_2b_3/q, b_1b_2a_3q)\big]. \notag \\
\end{align}
\end{cor}
We can rewrite the above identity. For $a_1, a_2, a_3 \neq 0$,
\begin{align}
&\langle-a_1;q\rangle_\infty\langle-a_2;q\rangle_\infty\langle-a_3;q^2\rangle_\infty\notag
\\
=&\langle-a_1{a_2}^{-1}q;
q^2\rangle_\infty\notag \\
&\times\big[\langle-a_1{a_2}a_3; q^4\rangle_\infty
\langle-{a_1}a_2{a_3}^{-1}q^2;
q^4\rangle_\infty+a_1a_2\langle-a_1{a_2}a_3q^2;
q^4\rangle_\infty\langle-{a_1}a_2{a_3}^{-1}q^4;
q^4\rangle_\infty\big]\notag
\\
&+a_1\langle-a_1{a_2}^{-1}q^2;
q^2\rangle_\infty\notag \\
&\times\big[\langle-a_1{a_2}a_3q;
q^4\rangle_\infty\langle-{a_1}a_2{a_3}^{-1}q^3;
q^4\rangle_\infty+a_3\langle-a_1{a_2}a_3q^3;
q^4\rangle_\infty\langle-{a_1}a_2{a_3}^{-1}q;
q^4\rangle_\infty\big].
\end{align}

Letting $a_3=-a_1a_2$ in Corollary \ref{cor3case2}, then replacing
$a_1$ with $-a_1$, $a_2$ with $-a_2$, and simplifying, we deduce the
next corollary.
\begin{cor} \label{3case21}
\begin{align}
&\langle a_1;q\rangle_\infty\langle a_2;q\rangle_\infty\langle
a_1a_2;q^2\rangle_\infty\notag
\\
=&\langle q^2; q^4\rangle_\infty\langle-a_1{a_2}^{-1}q;
q^2\rangle_\infty\langle {a_1}^2{a_2}^2; q^4\rangle_\infty
\notag \\
&+a_1\langle q; q^4\rangle_\infty\langle-a_1{a_2}^{-1}q^2;
q^2\rangle_\infty\big[{a_1}a_2\langle{a_1}^2{a_2}^2q^3;
q^4\rangle_\infty -\langle{a_1}^2{a_2}^2q; q^4\rangle_\infty\big].
\end{align}
\end{cor}
Letting $a_2=-1$ in Corollary \ref{3case21} and change the notation
from $a_1$ to $a$, we deduce the next result.
\begin{cor}
\begin{align}
&2(a;q)_\infty({a}^{-1}q;q)_\infty(-a;q^2)_\infty(-{a}^{-1}q^2;q^2)_\infty
\notag \\
=&(a^2; q^4)_\infty({a}^{-2}q^4; q^4)_\infty(aq;
q^2)_\infty({a}^{-1}q; q^2)_\infty
\notag \\
&-(aq^2; q^2)_\infty({a}^{-1}; q^2)_\infty(q;
q^2)_\infty(-q^2, q^2)^2_\infty\notag \\
&\times\big[a^2(a^2q^3; q^4)_\infty({a}^{-2}q; q^4)_\infty
 +a(a^2q;
q^4)_\infty({a}^{-2}q^3; q^4)_\infty\big].
\end{align}
\end{cor}
Next, we study the transformation matrix
$$
B=\left[
\begin{array}{rrrr}
1 & 1 & 0 & -1  \\
1 & -1 & 0 & 1  \\
0 & 1 & -1 & 1  \\
0 & 1 & 1 & 1
\end{array}
\right].
$$
The set of its four columns of $B$ is an orthogonal set and $\det
B=8$. Foregoing all details, we have the following corollary.
\begin{cor} \label{4case1}
If we let $q=a_ib_i$, where $i=1, 2, 3, 4$, then
\begin{align} \label{abcdefgh}
\prod _{i=1}^4 f(a_i,b_i)=&\sum _{r=0}^{7}
{a_1}^{\frac{r^2+r}{2}}{b_1}^{\frac{r^2-r}{2}}f(a_1a_2q^r,
b_1b_2q^{-r})f(a_1b_2a_3a_4q^{r},b_1a_2b_3b_4q^{-r})
\notag \\
&\times f(b_3a_4q^{r}, a_3b_4q^{-r})f(b_1a_2a_3a_4q^{r},
a_1b_2b_3b_4q^{-r}).
\end{align}
\end{cor}
Corollary \ref{cor3case2} can be shown as a special case of
Corollary \ref{4case1}. We omit the details here.

Corollary \ref{4case1} can be rewritten as follows. For $a_1, a_2,
a_3, a_4 \neq 0$,
\begin{align}
&\prod_{i=1}^4\langle-a_i;q\rangle_\infty\notag
\\
=&\langle-a_1{a_2}; q^2\rangle_\infty\langle-a_1{a_2}^{-1}a_3a_4q;
q^4\rangle_\infty\langle-{a_3}^{-1}{a_4}q;
q^2\rangle_\infty\langle-a_1^{-1}{a_2}{a_3}{a_4}q;
q^4\rangle_\infty\notag
\\
&+a_1\langle-a_1{a_2}q;
q^2\rangle_\infty\langle-a_1{a_2}^{-1}a_3a_4q^2;
q^4\rangle_\infty\langle-{a_3}^{-1}{a_4}q;
q^2\rangle_\infty\langle-a_1^{-1}{a_2}a_3a_4;
q^4\rangle_\infty\notag
\\
&+a_2\langle-a_1{a_2}q;
q^2\rangle_\infty\langle-a_1{a_2}^{-1}a_3a_4;
q^4\rangle_\infty\langle-{a_3}^{-1}{a_4}q;
q^2\rangle_\infty\langle-a_1^{-1}{a_2}a_3a_4q^2;
q^4\rangle_\infty\notag
\\
&+a_3\langle-a_1{a_2}; q^2\rangle_\infty(-a_1{a_2}^{-1}a_3a_4q^2;
q^4\rangle_\infty(-{a_3}^{-1}{a_4};
q^2)_\infty(-a_1^{-1}{a_2}a_3a_4q^2; q^4)_\infty\notag
\\
&+a_1a_3\langle-a_1{a_2}q;
q^2\rangle_\infty\langle-a_1{a_2}^{-1}a_3a_4q^3;
q^4\rangle_\infty\langle-{a_3}^{-1}{a_4};
q^2\rangle_\infty\langle-a_1^{-1}{a_2}{a_3}{a_4}q;
q^4\rangle_\infty\notag
\\
&+a_2a_3\langle-a_1{a_2}q;
q^2\rangle_\infty\langle-a_1{a_2}^{-1}a_3a_4q;
q^4\rangle_\infty\langle-{a_3}^{-1}{a_4};
q^2\rangle_\infty\langle-a_1^{-1}{a_2}{a_3}{a_4}q^3;
q^4\rangle_\infty\notag
\\
&+a_3a_4\langle-a_1{a_2};
q^2\rangle_\infty\langle-a_1{a_2}^{-1}a_3a_4q^3;
q^4\rangle_\infty\langle-{a_3}^{-1}{a_4}q;
q^2\rangle_\infty\langle-a_1^{-1}{a_2}{a_3}{a_4}q^3;
q^4\rangle_\infty\notag
\\
\label{product4theta1} &+{a_4}^{-1}q\langle-a_1{a_2};
q^2\rangle_\infty\langle-a_1{a_2}^{-1}a_3a_4q;
q^4\rangle_\infty\langle-{a_3}^{-1}{a_4}q;
q^2\rangle_\infty\langle-a_1^{-1}{a_2}{a_3}{a_4}q;
q^4\rangle_\infty.
\end{align}

From \eqref{abcdefgh}, we deduce the following.
\begin{cor}
Let $q=a_ib_i$, $(i=1, 2, 3, 4)$, then
\begin{align*}
&\prod _{i=1}^4 f(a_i,b_i)+\prod _{i=1}^4 f(-a_i,-b_i)\notag \\
=&2f(a_1a_2, b_1b_2)f(a_1b_2a_3a_4,b_1a_2b_3b_4)f(a_3b_4, a_4b_3)f(a_1b_2b_3b_4, b_1a_2a_3a_4)\notag \\
&+2a_1a_4f(a_1a_2q, b_1b_2/q)f(a_1b_2a_3a_4q^2,b_1a_2b_3b_4/q^2)f(a_3b_4/q, a_4b_3q)f(a_1b_2b_3b_4, b_1a_2a_3a_4)\notag \\
&+2a_2a_4f(a_1a_2q, b_1b_2/q)f(a_1b_2a_3a_4,b_1a_2b_3b_4)f(a_3b_4/q, a_4b_3q)f(a_1b_2b_3b_4/q^2, b_1a_2a_3a_4q^2)\notag \\
&+2a_3a_4f(a_1a_2,
b_1b_2)f(a_1b_2a_3a_4q^2,b_1a_2b_3b_4/q^2)f(a_3b_4, a_4b_3)f(a_1b_2b_3b_4/q^2, b_1a_2a_3a_4q^2).\\
&\prod _{i=1}^4 f(a_i,b_i)-\prod _{i=1}^4 f(-a_i,-b_i)\notag \\
=&2a_1f(a_1a_2q, b_1b_2/q)f(a_1b_2a_3a_4q,b_1a_2b_3b_4/q)f(a_3b_4, a_4b_3)f(a_1b_2b_3b_4q, b_1a_2a_3a_4/q)\notag \\
&+2a_2f(a_1a_2q, b_1b_2/q)f(a_1b_2a_3a_4/q,b_1a_2b_3b_4q)f(a_3b_4, a_4b_3)f(a_1b_2b_3b_4/q, b_1a_2a_3a_4q)\notag \\
&+2a_4f(a_1a_2, b_1b_2)f(a_1b_2a_3a_4q,b_1a_2b_3b_4/q)f(a_3b_4/q, a_4b_3q)f(a_1b_2b_3b_4/q, b_1a_2a_3a_4q)\notag \\
&+2b_3f(a_1a_2, b_1b_2)f(a_1b_2a_3a_4/q,b_1a_2b_3b_4q)f(a_3b_4/q,
a_4b_3q)f(a_1b_2b_3b_4q, b_1a_2a_3a_4/q).
\end{align*}
\end{cor}
M. Hirschhorn \cite{hir1} considered the transformation matrix
$$
A=\left[
\begin{array}{rrrr}
1 & 0 & 1 & 1  \\
0 & 1 & 1 & -1  \\
1 & 1 & -1 & 0  \\
1 & -1 & 0 & -1
\end{array}
\right].
$$
and derived a generalization of Winquist's identity.

If we consider the transformation matrix as a Hadamard matrix with
determinant 16
$$
B=\left[
\begin{array}{rrrr}
1 & 1 & 1 & 1  \\
1 & 1 & -1 & -1  \\
1 & -1 & -1 & 1  \\
1 & -1 & 1 & -1
\end{array}
\right],
$$
we deduce the following corollary corresponding to its ECS.
\begin{cor}
Let $q=a_ib_i$, where $i=1, 2, 3, 4$, then
\begin{align} \label{result 42}
\prod _{i=1}^4 f(a_i,b_i)=&\sum _{r=0}^{15} {a_1}^{\frac{r^2+r}{2}}{b_1}^{\frac{r^2-r}{2}}f(a_1a_2a_3a_4q^r,
b_1b_2b_3b_4q^{-r})f(a_1a_2b_3b_4q^r,b_1b_2a_3a_4q^{-r})\notag \\
&\times f(a_1b_2b_3a_4q^r,b_1a_2a_3b_4q^{-r})f(a_1b_2a_3b_4q^r,
b_1a_2b_3a_4q^{-r}).
\end{align}
\end{cor}
Replacing $a_3$ with $a_1$, $a_4$ with $-a_2$, and then ${a_2}^2$
with $a_2$ in \eqref{result 42}, we deduce the following corollary.
\begin{cor} \label{3theta1}
\begin{align}
& \langle-a_1;q\rangle^2_\infty\langle-a_2;q^2\rangle_\infty\notag
\\
=&\varphi(q)\big[\langle-{a_1}^2{a_2};q^{4}\rangle_\infty\langle-{a_1}^2{a_2}^{-1}q^{2};q^{4}\rangle_\infty+{a_1}^2\langle{a_1}^2{a_2}q^{2};
q^{4}\rangle_\infty\langle{a_1}^2{a_2}^{-1}q^{4};q^{4}\rangle_\infty\big]\notag
\\
\label{result 43} &+2\psi(q^2)\big[
a_1\langle-{a_1}^2{a_2}q;q^{4}\rangle_\infty\langle{a_1}^2{a_2}^{-1}q^{3};q^{4}\rangle_\infty+
a_1a_2\langle-{a_1}^2{a_2}q^{3};q^{4}\rangle_\infty\langle{a_1}^2{a_2}^{-1}q;q^{4}\rangle_\infty\big].
\end{align}
\end{cor}
We can rewrite \eqref{result 43} as follows. For $ab=q$,
$cd=q^2$,
\begin{align}
&f^2(a, b)f(c, d)\notag \\
=&\varphi(q)\big[(a^2c, b^2d)f(a^2d,
b^2c)+a^2f(a^2cq^2, b^2d/q^2)f(a^2dq^2, b^2c/q^2)\big]\notag \\
\label{result 431} &+2\psi(q^2)\big[(af(a^2cq, b^2d/q)f(a^2dq,
b^2c/q)+bf(a^2c/q, b^2dq)f(a^2d/q, b^2cq)\big].
\end{align}
In \eqref{result 431}, setting $c=-a^2$ and $d=-b^2$. By
\eqref{entry 311}, $f({a}^{-1}b, -a^2)=f(-a^{-1}b^3,
-a^5b)+a^{-1}bf(-a^3b^{-1}, -ab^{5})$. We deduce the following
identity.
\begin{cor}
For $ab=q$,
\begin{align}
f^2(a,b)f(-a^2, -b^2)=&\varphi(q)\varphi(-q^2)f(-a^4,
-b^4)+2a\psi(-q)\psi(q^2)f({a}^{-1}b, -a^2).
\end{align}

\end{cor}
\noindent

Similarly, setting $c=a^2$ and $d=b^2$ in \eqref{result 431}, we
have the next corollary after simplification.
\begin{cor}
For $ab=q$,
\begin{align}
&f^2(a, b)f(a^2, b^2)\notag \\
=&\varphi(q)\big[f(a^4, b^4)\phi(a^2b^2)+2a^2f(a^{-2}b^2,
a^4q^2)\psi( q^4)\big]+2a\psi(q)\psi(q^2)f({a}^{-1}b, a^2).
\end{align}
\end{cor}

Using \eqref{result 431}, we deduce the following corollary.
\begin{cor}
For $ab=q$, $cd=q^2$,
\begin{align}
&[f^2(a, b)+f^2(-a, -b)]f(c, d)\notag \\
\label{sum 42}
 =&2\varphi(q)\big[f(a^2c, b^2d)f(a^2d,
b^2c)+a^2f(a^{-2}d, a^2cq^2)f(a^{-2}c, a^2dq^2)\big].
\end{align}
\end{cor}
We rewrite the above identity in its equivalent form.
\begin{align}
&\big[(a;q)^2_\infty(a^{-1}q;q)^2_\infty+(-a;q)^2_\infty(-a^{-1}q;q)^2_\infty\big](b;q^2)_\infty(b^{-1}q^2;q^2)_\infty\notag
\\
=&2(-q;q)^4_\infty(a^2b;q^{4})_\infty(a^{-2}b^{-1}q^{4};q^{4})_\infty(a^2b^{-1}q^{2};q^{4})_\infty(a^{-2}bq^{2};q^{4})_\infty\notag
\\
\label{sum 42} &+2a^2(-q;q)^4_\infty (a^2bq^{2};
q^{4})_\infty(a^{-2}b^{-1}q^2;q^{4})_\infty(a^2b^{-1}q^{4};q^{4})_\infty(a^{-2}b;
q^{4})_\infty.
\end{align}
Setting $b=a^2$ in \eqref{sum 42}, we deduce the following corollary
after simplification.
\begin{cor}
\begin{align*}
f^2(a, b)+f^2(-a, -b)=2f(a^2, b^2)\varphi(ab).
\end{align*}
\end{cor}
This is (v) of Entry 30 in \cite[p.~46]{Ber1}.

We can rewrite the above identity as
\begin{align*}
(a;q)^2_\infty(a^{- 1}q;q)^2_\infty+(-a;q)^2_\infty(-a^{-1}q;q)^2_\infty=2(-q;q)^2_\infty(-q;q^2)^2_\infty(-a^2;q^{2})_\infty(-a^{-2}q^{2};q^{2})_\infty.
\end{align*}
We can find other interesting cases by exploring different integer
matrix exact covering systems. For example, if we choose
$$
B=\left[
\begin{array}{rrrrr}
1 & 0 & -1 & -1 & 1 \\
1 & 0 & 0 & 2 & 0 \\
0 & 1 & 1 & 0 & 1\\
0 & -1 & 1 & 0 & 1\\
-1 & 0 & -1 & 1 & 1
\end{array}
\right],
$$
we can write a product of five theta functions as the linear
combination of twenty four products of five theta functions. We omit
the details here.\\
\textbf{Acknowledgement:} This work grow out of the author's Ph.D.
dissertation. The author is very grateful to his advisor Professor
Bruce Berndt for his helpful guidance.

\end{document}